\documentclass[11pt, leqno]{amsart} 
\usepackage{amssymb,amscd,amsfonts,amsbsy}
\usepackage{latexsym}
\usepackage{exscale}
\usepackage{graphicx} 
\usepackage{caption} 
\usepackage{amsmath,amsthm,amsfonts}
\usepackage{mathrsfs}
\usepackage{xcolor} 
\usepackage{esint} 
\usepackage{stmaryrd}
\usepackage{pifont}
\usepackage{tikz}
\usepackage{subcaption}
\usepackage{asymptote}
\usepackage{pgfplots}
\usepackage{subcaption}
\pgfplotsset{compat=1.18}
\usetikzlibrary{patterns}
\usepackage[utf8]{inputenc}

\calclayout

\allowdisplaybreaks

\makeatletter
\@namedef{subjclassname@2020}{%
  \textup{2020} Mathematics Subject Classification}
\makeatother

\theoremstyle{plain}
\newtheorem{theorem}[equation]{Theorem}
\newtheorem{lemma}[equation]{Lemma}

\newtheorem{proposition}[equation]{Proposition}

\theoremstyle{definition}

\newtheorem{remark}[equation]{Remark}

\numberwithin{equation}{section} 

\DeclareMathOperator*{\esssup}{ess\,sup}

\begin{document}
\author{Mingming Cao}
\address{Mingming Cao\\
Institute for Advanced Study in Mathematics\\
Harbin Institute of Technology\\
Harbin 150001\\
People's Republic of China} \email{mingming.cao@hit.edu.cn}

\author{Honghai Liu}
\address{Honghai Liu\\
School of Mathematics and Information Science\\
Henan Polytechnic University\\
Jiaozuo 454000\\
People's Republic of China} \email{hhliu@hpu.edu.cn}

\author{Zengyan Si}
\address{Zengyan Si\\
School of Mathematics and Information Science\\
Henan Polytechnic University\\
Jiaozuo 454000\\
People's Republic of China} \email{zengyan@hpu.edu.cn}

\author{Haibin Yan}
\address{Haibin Yan\\
School of Mathematics and Information Science\\
Henan Polytechnic University\\
Jiaozuo 454000\\
People's Republic of China} \email{hbyan@home.hpu.edu.cn}

\thanks{The first author was partially supported by grants RYC2021-032600-I and PID2022-141354NB-I00. All authors were supported  by NSF of China grant 12571101.}

\year=2025 \month=07 \day=13

\date{\today}




\begin{abstract}
Let $M_\psi$ be the maximal operator associated to $\psi(x)=-\ln(1-|x|^2)$. We obtain dimension-free bounds on $L^p(\mathbb R^d)$ for $M_\psi$, $1<p\le\infty$. We show that the weak $L^1$ estimate does not hold for $M_\psi$, but does hold for the corresponding lacunary maximal operator $M_\psi^D$. An weak $L\log L$ type estimate for $M_\psi$ is obtained, which seems to be the right substitution for the weak $L^1$ estimate. Estimates for $M_\psi$  are applied to dimension-free estimates for maximal Riesz transforms obtained by Mateu, Verdera, Orobitg and P\'{e}rez. 
\end{abstract}

\title{Dimension-free bounds for maximal operators associated to log kernel and applications} 

\maketitle
\tableofcontents

\section{Introduction}
Given a function $g$ on $\mathbb R^d$ and $t>0$, we denote by $g_t$ the function $g_t(x)=t^{-d}g(t^{-1}x)$. Let $\psi(x)=-\ln(1-|x|^2)$,  the maximal function $M_\psi$ is given by
\begin{align*}
M_\psi f(x)=\sup_{t>0}|\psi_t\ast f(x)|.
\end{align*}
Note that $\psi$ is positive, unbounded, integrable and supported in open unit ball $B(0,1)$, and it has not an integrable radial decreasing majorant, that is, there is no nonnegative decreasing function $K_0$ on $[0,\infty)$  such that
$\psi(x)\le K_0(|x|)$ and $\|K_0(|\cdot|)\|_{L^1(\mathbb R^d)}<\infty$.
\par
In 1976 Stein \cite{S76} introduced the spherical maximal means of order $\alpha$
\begin{align*}
\mathscr M^\alpha f(x)=\sup_{t>0}\left|\frac1{\Gamma(\alpha)}\int_{|y|\le1}\left(1-|y|^2\right)^{\alpha-1}f(x-ty)dy\right|,
\end{align*}
where $\alpha>0$ and $\Gamma(\alpha)$ is the gamma function.
Stein showed that
\begin{align}\label{sphmmalpha}
\|\mathscr M^\alpha f\|_{L^p(\mathbb R^d)}\le C_{p,d,\alpha}\|f\|_{L^p(\mathbb R^d)}
\end{align}
in the following circumstances
\begin{align*}
\text{if}\ \ 1<p\le2,\ \ \text{when}\ \alpha>1-d+\frac dp
\end{align*}
or
\begin{align}\label{Srangealpha}
\text{if}\ \ 2\le p\le\infty,\ \ \text{when}\ \alpha>\frac{2-d}p.
\end{align}
However, the range $\alpha$ in \eqref{Srangealpha} is not optimal, see \cite{LSSY24,MYZ17}.
\par
When $\alpha=1$, $\mathscr M^\alpha$ corresponds to the Hardy-Littlewood maximal operator associated to the Euclidean ball, \eqref{sphmmalpha} holds for all $1<p\le\infty$, one can replace the bound $C_{p,d,1}$ by a bound $C_p$ independent of the dimension $d$, see \cite{S82,SS83}. For $\alpha\in(0,1)$, $\left(1-|y|^2\right)^{\alpha-1}$ and $-\ln\left(1-|y|^2\right)$ blow up in $y$ near the unit sphere $\mathbb S^{d-1}$, and there is a $\varepsilon_\alpha>0$ such that
\begin{align*}
\left(1-|y|^2\right)^{\alpha-1}>-\ln\left(1-|y|^2\right),\ \ 1-\varepsilon_\alpha<|y|<1.
\end{align*}
When $\alpha\in(0,1)$, \eqref{sphmmalpha} holds for $p>\frac{d}{d+\alpha-1}$. The range of $p$ in \eqref{sphmmalpha} approches to $(1,\infty]$ as $\alpha$ goes to $1$. 
\par
A natural question can be raised: Is $M_\psi$ bounded on $L^p(\mathbb R^d)$ for $1<p\le\infty$? Does $M_\psi$ have $L^p$ bounds independent of the dimension $d$? Answers are provided by the following result.
\begin{theorem}\label{dfppe}
Let $d\in\mathbb N$. For each $1<p\le\infty$, there is a constant $C_p$ independent of the dimension $d$ such that
\begin{align*}
\|M_\psi f\|_{L^p(\mathbb R^d)}\le C_p\|\psi\|_{L^1(\mathbb R^d)}\|f\|_{L^p(\mathbb R^d)}.
\end{align*}
\end{theorem}
One may wonder what may happen in the limiting case $p=1$, we answer this question by the next result.
\begin{theorem}\label{notw11}
$M_\psi$ is not of weak type $(1,1)$, i.e. for any positive constant $C$, there exists a function $f\in L^1(\mathbb R^d)$ such that
\begin{align*}
\|M_\psi f\|_{L^{1,\infty}(\mathbb R^d)}\ge C\|f\|_{L^1(\mathbb R^d)}.
\end{align*}
However, $M_\psi$ is of weak type $L\log L$, i.e.
there is a positive constant $C$ independent of the dimension $d$ such that
\begin{align}\label{epe}
\left|\left\{x\in\mathbb R^d:M_\psi f(x)>\lambda\right\}\right|\le C\int_{\mathbb R^d}\Phi\left(\frac{|f(x)|}\lambda\right)dx,\ \ \lambda>0,
\end{align}
where $\Phi(s)=s\left(1+\log^+s\right)$.
\end{theorem}

Let $\mu$ be a compactly supported finite Borel measure, the maximal operator and the lacunary maximal operator are given by 
\begin{align*}
M_\mu f(x)=\sup_{t>0}\int_{\mathbb R^d}|f(x-ty)d\mu(y)|
\end{align*}
and
\begin{align*}
M^D_\mu f(x)=\sup_{k\in\mathbb Z}\int_{\mathbb R^d}|f(x-2^ky)d\mu(y)|,
\end{align*}
respectively. Obviously, $M_\mu$ controls $M^D_\mu$.  For many singular measures it has been observed that $M_\mu$ does not satisfy some estimates but $M^D_\mu$ does. For instance, when $\mu$ is the normalized Lebesgue measure on the sphere $\mathbb S^{d-1}$, $M_\mu$ is the spherical maximal function, which is not bounded on $L^p(\mathbb R^d)$ for $1<p\le\frac d{d-1}$, the lacunary spherical maximal function is bounded from $L^p(\mathbb R^d)$ to  $L^p(\mathbb R^d)$ for all $1<p\le\infty$ when $d\ge2$, see \cite{S76} and \cite{Calderon79}. Theorem \ref{notw11}  implies that $M_\psi$ does not map $L^1(\mathbb R^d)$ to $L^{1,\infty}(\mathbb R^d)$. Whether the corresponding lacunary maximal operator $M_\psi^D$ maps $L^1(\mathbb R^d)$ to $L^{1,\infty}(\mathbb R^d)$? We obtain the following result.
\begin{theorem}\label{dyad w11}
$M^D_\psi$ is  of weak type $(1,1)$, i.e. there is a positive constant $C_d$ dependent on the dimension $d$ such that
\begin{align*}
\lambda\left|\left\{x\in\mathbb R^d:M^D_\psi f(x)>\lambda\right\}\right|\le C_d\|f\|_{L^1(\mathbb R^d)},\ \ \lambda>0.
\end{align*}
\end{theorem}
\par
Dimension-free estimates are useful when one is interested in generalizing results from the finite dimensional to the infinite dimensional setting. The starting point of research was the work of Stein \cite{S82}. Let $G$ be a non-empty convex symmetric body in $\mathbb R^d$,  the Hardy-Littlewood maximal operator $M_G$ associated to convex set $G$ is defined by
\begin{align*}
M_Gf(x)=\sup_{t>0}\frac1{|G_t|}\int_{G_t}f(x-y)dy, \ \ f\in L_{loc}^1(\mathbb R^d),
\end{align*}
where $G_t=\{y\in\mathbb R^d:t^{-1}y\in G\}$. Given $1\le q<\infty$, $\|x\|_q:=(|x_1|^q+|x_2|^q+\cdots+|x_d|^q)^{1/q}$, and $\|x\|_\infty:=\max_{1\le i\le d}\{|x_i|\}$. For $1\le q\le\infty$, define balls $B_q=\{x\in \mathbb R^d:\|x\|_q\le1\}$. In \cite{S82}, Stein obtained the dimension-free bounds on $L^p(\mathbb R^d)$ for $M_{B_2}$ with $p\in(1,\infty]$. Stein's result was extended to $M_G$ with general symmetric convex body $G$ by Bourgain \cite{B86a,B86b} and Carbery \cite{Car86} when $p\in(3/2,\infty]$. The question about the dimension-free bounds for general symmetric convex body $G$ for $p\in(1,3/2]$ remains still open. Specially, for $G=B_q$, $1\le q<\infty$, M\"{u}ller \cite{M90} proved dimension-free estimates on $L^p(\mathbb R^d)$ for $p\in(1,\infty]$. For $G=B_\infty$,  the question about dimension-free bounds on $L^p(\mathbb R^d)$ for $M_{B_\infty}$ with $(1,3/2]$ was left open until 2014,  Bourgain \cite{B14} showed that dimension-free estimates on $L^p(\mathbb R^d)$ hold for all $p\in(1,\infty]$. Dimension-free estimates in Theorem \ref{dfppe} may be regarded as generalizations for dimension-free $L^p$ estimates for the Hardy-Littlewood maximal operator associated to convex sets. 
\par
Stein raise a general question in \cite{S83} that ''Can one find an appropriate infinite-dimensional formation of (that part of ) harmonic analysis in $\mathbb R^d$, which displays in a natural way the above uniformity in $d$?''. In fact, dimension-free $L^p$ estimates also hold for Riesz transforms. 
\par
For $j=1,2,\cdots,d$, let $K^j$ be the kernel
 \begin{align*}
 K^j(x)=c_d\chi_{\{|x|>1\}}(x)\frac{x_j}{|x|^{d+1}}\ \ \text{with}\ \ c_d=\frac{\Gamma(\frac{d+1}2)}{\pi^{\frac{d+1}2}}.
 \end{align*}
The $j$th Riesz transform $R_j$ is given by
\begin{align*}
R_jf(x)=\lim_{t\rightarrow0^+}K^j_t\ast f(x),\ \ f\in\mathscr S(\mathbb R^d).
\end{align*}
In \cite{S83}, Stein proved that 
\begin{align*}
\left\|\left(\sum_{j=1}^d|R_jf|^2\right)^{1/2}\right\|_{L^p(\mathbb R^d)}\le C_p\|f\|_{L^p(\mathbb R^d)},\ \ 1<p<\infty,
\end{align*}
with $C_p$ independent of $d$. The analogue has been also proved for higher order Riesz transforms, see \cite[Th\'{e}or\`{e}me 2]{DR85}.
\par
The maximal truncated Riesz transform $R_j^\ast$ and the lacunary maximal truncated Riesz transforms  are defined by
\begin{align*}
R_j^\ast f(x)=\sup_{t>0}|K^j_t\ast f(x)|
\end{align*}
and 
\begin{align*}
R_j^{\ast,D} f(x)=\sup_{k\in\mathbb Z}|K^j_{2^k}\ast f(x)|,
\end{align*}
respectively. In \cite{MV06}, Mateu and Verdera showed that there exists a constant $C_{p,d}$ such that
\begin{align*}
\|R_j^\ast f\|_{L^p(\mathbb R^d)}\le C_{p,d}\|R_jf\|_{L^p(\mathbb R^d)}
\end{align*}
for all $f\in L^p(\mathbb R^d)$ with $1<p<\infty$. Recently, using Fourier transform estimates and square function techniques, Kucharski and Wr\'{o}bel \cite{KW23} improved above estimates for $p=2$ and obtained the following dimension-free $L^2$ estimate
\begin{align}\label{dfRj}
\|R_j^\ast f\|_{L^2(\mathbb R^d)}\le 2\cdot 10^8\|R_jf\|_{L^2(\mathbb R^d)},\ \  j=1,2,\cdots,d.
\end{align}
 Later, applying the complex method of rotations of Iwaniec and Martin \cite{IM96}, Kucharski, Wr\'{o}bel and Zienkiewicz obtained dimension-free $L^p$ estimates for the vector of  maximal Riesz transforms for all orders in terms of Riesz transforms with $1<p<\infty$, which includes the case when Riesz transforms are first order, see \cite{KWZ22,KWZ23}. Very recently, Liu, Melentijevi\'{c} and Zhu \cite{LMZ24} proved \eqref{dfRj} also holds for $p\ge2$ and provided dimension-free $L^p$ bounds which are not greater than $(2+1/\sqrt2)^{2/p}$.
\par
Applying dimension-free $L^p$ estimates and endpoint estimates for $M_\psi$, we provide new proofs  of dimension-free $L^p$ estimates and endpoint estimates for the maximal Riesz transform in terms of the Riesz transform.
\begin{theorem}\label{Rdfppe}
Let $d\in\mathbb N$. For each $1<p<\infty$, there is a constant $C_p$ independent of the dimension $d$ such that
\begin{align*}
\| R^\ast_jf\|_{L^p(\mathbb R^d)}\le C_p\|R_jf\|_{L^p(\mathbb R^d)},\ j=1,2,\cdots,d.
\end{align*}
\end{theorem}
In the limiting case $p=1$, we obtain the following results.
 \begin{theorem} \label{endpoint estimates Rj}
 Let $d\in\mathbb N$.  For every $1\le j\le d$ and constant $C>0$, there exists a function $f$ in $H^1(\mathbb R^d)$ such that
\begin{align*}
\|R^\ast_jf\|_{L^{1,\infty}(\mathbb R^d)}\ge C\|R_jf\|_{L^1(\mathbb R^d)}.
\end{align*}
However, there is a constant $C_d>0$ such that
\begin{align*}
\lambda\left|\left\{x\in\mathbb R^d:R^{D,\ast}_jf(x)>\lambda\right\}\right|\le C_d\|R_jf\|_{L^1(\mathbb R^d)},\ \ \lambda>0,
\end{align*}
and
\begin{align*}
\left|\left\{x\in\mathbb R^d:R^\ast_jf(x)>\lambda\right\}\right|\le C_d\int_{\mathbb R^d}\Phi\left(\frac{|R_jf(x)|}\lambda\right)dx,\ \ \lambda>0,
\end{align*}
where $\Phi(s)=s\left(1+\log^+s\right)$.
\end{theorem}
\par
The rest of the paper is organized as follows. In Section 2, we obtain $L^p$ estimates and endpoint estimates for $M_\psi$ with bounds dependent on the dimension $d$ by controlling the maximal operator $M_\psi f$ by $M^2f$, where $M^2=M\circ M$ its iteration. Theorem \ref{ppe},Theorem \ref{notw11} and Theorem \ref{dyad w11} are proved in Section 2. 
After that, we present the proof of Theorem \ref{dfppe} in Section 3 by using the spherical maximal operator. After that, we apply our main results to prove some known estimates for maximal Riesz transforms in terms of Riesz transforms in Section 4. 
\par
Throughout the whole paper $d$ denotes the dimension. For $k\in\mathbb N$ and $k\ge2$, we denote by $\nu_k$ the volume of the unit ball in $\mathbb R^k$ and by $\omega_{k-1}$ the surface area of the unit sphere $\mathbb S^{k-1}$. $C>0$ stands for an universal constant whose value may vary from occurrence to occurrence and never depend on the dimension, $\lceil C\rceil$ is the smallest integer greater than $C$. Symbols $C_\Delta$ and $C(\Delta)$ stand for constants that possibly depends on $\Delta$.

\section{Estimates for $M_\psi$}\label{est Mpsi}
The goal of this section is to obtain some estimates for $M_\psi$ with bounds dependent on the dimension $d$.
\subsection{$L^p$ estimates for $M_\psi$}
To prove Theorem \ref{dfppe}, we first establish $L^p$ estiamtes for $M_\psi$ with bounds dependent on the dimension $d$. We obtain the following result.
\begin{theorem}\label{ppe}
For $1<p\le\infty$, there is a constant $C_p$ independent of the dimension $d$ such that
\begin{align}\label{dfm1}
\|M_\psi f\|_{L^p(\mathbb R^d)}\le C_p4^d\nu_d\ln d\|f\|_{L^p(\mathbb R^d)}.
\end{align}
\end{theorem}
\begin{remark}
For any $a>0$, $a^d\omega_{d-1}$ approaches $0$ as $d$ goes to infinity. It seems \eqref{dfm1} is a dimension free $L^p$ estimate, but we should note that $\|\psi\|_{L^1(\mathbb R^d)}$ is not $1$ and depedent on the dimension $d$. Indeed, 
\begin{align}\label{psil1}
  \|\psi\|_{L^1(\mathbb R^d)}&=-\omega_{d-1}\int_0^1\ln(1-r^2)r^{d-1}dr\le\frac{\omega_{d-1}\left(\gamma+\Psi(d+1)+1\right)}d\le \frac{3\omega_{d-1}\ln(d+1)}d,
\end{align}
where $\gamma=0.57721566\cdots$ is the Euler-Mascheroni constant, $\Psi$ is the Digamma function, and we have used the following inequality \cite{AQ97}
\begin{align*}
\ln s-\frac1s<\Psi(s)<\ln s-\frac1{2s},\ \ s>0.
\end{align*}
Observe that $\|\psi\|_{L^1(\mathbb R^d)}$ decays faster than $4^d\omega_{d-1}$ to $0$. This indicates that Theorem \ref{ppe} has to be improved further to get Theorem \ref{dfppe}.
\end{remark}
In order to prove Theorem \ref{ppe}, let us record some facts on Orlicz spaces, see \cite{O88,RR91}. 
A function $\Phi:[0,\infty)\rightarrow[0,\infty)$ is a Young function if it is continuous, convex, and increasing satisfying $\Phi(0)=0$ and $\Phi(t)\rightarrow\infty$ as $t\rightarrow\infty$. We define the $\Phi$-average of a function $f$ over a ball $B$ by means of the following Luxemburg norm:
\begin{align*}
\|f\|_{\Phi,B}=\inf\left\{\lambda>0:\frac1{|B|}\int_B\Phi\left(\frac{|f(y)|}{\lambda}\right)dy\le 1\right\}.
\end{align*}
In this paper, we consider Young functions $\Phi(t)=t\left(1+\log^+t\right)$ and $\Psi(t)=e^t-1$. Corresponding averages will be denoted by $\|\cdot\|_{\Phi,B}=\|\cdot\|_{L\log L,B}$ and $\|\cdot\|_{\Psi,B}=\|\cdot\|_{\exp L,B}$. The following H\"{o}lder's inequality for Orlicz spaces holds
\begin{align}\label{gHi}
\frac1{|B|}\int_B|f(y)g(y)|dy\le \|f\|_{L\log L,B}\|g\|_{\exp L,B}.
\end{align}
The Orlicz maximal operator is defined by
\begin{align*}
M_{L\log L}f(x)=\sup_{B\ni x}\|f\|_{\Phi,B}.
\end{align*}
It is proved in \cite{CP94} that
\begin{align}\label{logm2}
M_{L\log L}f(x)\le C4^dM^2f(x),
\end{align}
where $C$ is independent of the dimension $d$, see also \cite[Lemma 3.5.4]{G14M}.
Using \eqref{logm2}, we get the following pointwise estimate that will be proved monmentarily.
\begin{proposition}\label{M2}
There is a constant $C$ independent of the dimension $d$ such that
\begin{align*}
M_\psi f(x)\le C4^d\nu_d\ln d\ M^2f(x),\ \ x\in\mathbb R^d.
\end{align*}
\end{proposition}
\begin{proof}[Proof of Theorem \ref{ppe}]
Assuming Proposition \ref{M2} for the moment, we can quickly prove Theorem \ref{ppe}. Indeed, Theorem \ref{ppe} follows from Proposition \ref{M2} and the fact  (see \cite[Theorem 13]{S82},\cite[Theorem A]{SS83}) that 
\begin{align*}
\|Mf\|_{L^p(\mathbb R^d)}\le C_p\|f\|_{L^p(\mathbb R^d)},\ \ 1<p\le\infty,
\end{align*}
 where $C_p$ is independent of dimension $d$.
 \end{proof}
 It remains to prove Proposition \ref{M2}, which is also crucial for the weak $L\log L$ type estimate. 
\begin{proof}[Proof of Proposition \ref{M2}]
By translating and dilating one reduces the proof of Proposition \ref {M2} to
\begin{align*}
\psi\ast f(0)\le C4^d\nu_d\ln d\ M^2f(0).
\end{align*}
\par
Note that $\omega_{d-1}=d\nu_d$. For $1\le p<\infty$, $f\in L^p(\mathbb R^d)$,  \eqref{gHi} and \eqref{logm2} give
\begin{align*}
|\psi\ast f(0)|&=\left|\int_{B(0,1)}\left(\psi(y)-1\right)f(y)dy+\int_{B(0,1)}f(y)dy\right|\\
&\le\int_{B(0,1)}\left|\psi(y)-1\right|f(y)dy+|B(0,1)|Mf(0)\\
&\le |B(0,1)|\|\psi-1\|_{\exp L,B(0,1)}\|f\|_{L\log L,B(0,1)}+|B(0,1)|Mf(0)\\
&\le C\nu_d\ln d\ M_{L\log L}f(0)+\nu_dMf(0)\\
&\le C4^d\nu_d\ln d\ M^2f(0),
\end{align*}
where we have used the fact that
\begin{align}\label{psi-1}
\|\psi-1\|_{\exp L,B(0,1)}<C\ln d.
\end{align}
Indeed, we write
\begin{align*}
\|\psi-1\|_{\exp L,B(0,1)}&=\inf\left\{\lambda>0:\frac1{|B(0,1)|}\int_{B(0,1)}\left(e^{\frac{|\psi(y)-1|}{\lambda}}-1\right)dy\le 1\right\}\\
&=\inf\left\{\lambda>0:\frac1{|B(0,1)|}\int_{B(0,1)}e^{\frac{|\psi(y)-1|}{\lambda}}dy\le 2\right\}.
\end{align*}
\par
Observe that $-\ln(1-s^2)$ is strictly increasing on $[0,1)$, $\psi(0)=0$ and $\psi(y)$ approches $+\infty$ as $|y|$ tends to $1$ from the left. These facts imply that $\psi(y)-1<0$ when $y\in B(0,\sqrt{1-e^{-1}})$ and $\psi(y)-1\ge0$ when $y\in B(0,1)\setminus B(0,\sqrt{1-e^{-1}})$. It follows that
\begin{align*}
e^{\frac{|\psi(y)-1|}{\lambda}}=e^{\frac1\lambda}\chi_{B(0,\sqrt{1-e^{-1}})}(y)(1-|y|^2)^{\frac1{\lambda}}+e^{-\frac1\lambda}\chi_{B(0,1)\setminus B(0,\sqrt{1-e^{-1}})}(y)(1-|y|^2)^{-\frac1{\lambda}}.
\end{align*}
When $\lambda>2$, it is dominated by the following integrable function
\begin{align*}
e^{\frac12}\chi_{B(0,\sqrt{1-e^{-1}})}(y)+\chi_{B(0,1)\setminus B(0,\sqrt{1-e^{-1}})}(y)(1-|y|^2)^{-\frac12}.
\end{align*}
Lebesgue's dominated convergence theorem yields
\begin{align*}
\lim_{\lambda\rightarrow+\infty}\frac1{|B(0,1)|}\int_{B(0,1)}e^{\frac{|\psi(y)-1|}{\lambda}}dy=1.
\end{align*}
Thus, there is a $\lambda_0>0$ such that 
\begin{align*}
\frac1{|B(0,1)|}\int_{B(0,1)}e^{\frac{|\psi(y)-1|}{\lambda_0}}dy\le 2.
\end{align*}
Therefore, we have
\begin{align*}
\|\psi-1\|_{\exp L,B(0,1)}\le \lambda_0<\infty.
\end{align*}
Moreover, when $d\ge5$ and $\lambda=4\ln d$, by switching to polar coordinates, we observe that
\begin{align*}
\frac1{|B(0,1)|}\int_{B(0,1)}e^{\frac{|\psi(y)-1|}{\lambda}}dy&=e^{\frac1\lambda}d\int_0^{\sqrt{1-e^{-1}}}(1-r^2)^{\frac1{\lambda}}r^{d-1}dr+e^{-\frac1\lambda}d\int^1_{\sqrt{1-e^{-1}}}(1-r^2)^{-\frac1{\lambda}}r^{d-1}dr\\
&\le e^{\frac1\lambda}d\int_0^{\sqrt{1-e^{-1}}}r^{d-1}dr+e^{-\frac1\lambda}d^{\frac1\lambda}d\int^{\sqrt{1-d^{-1}}}_{\sqrt{1-e^{-1}}}r^{d-1}dr\\
&+e^{-\frac1\lambda}d\int^1_{\sqrt{1-d^{-1}}}(1-r^2)^{-\frac1{\lambda}}rdr\\
&\le e^{\frac1\lambda}(1-e^{-1})^{\frac d2}+e^{-\frac1\lambda}d^{\frac1\lambda}\left(1-(1-e^{-1})^{\frac d2}\right)+\frac{d^{\frac1\lambda}e^{-\frac1\lambda}\lambda}{2(\lambda-1)}\\
&\le2,
\end{align*}
which means 
\begin{align*}
\|\psi-1\|_{\exp L,B(0,1)}\le C\ln d.
\end{align*}
This completes the proof. 
\end{proof}
\begin{remark}
The classical John-Nirenberg inequality implies that for $b\in BMO$, there is a constant $C_d$ such that
\begin{align*}
\|b-b_Q\|_{\exp L,Q}\le C_d\|b\|_{BMO}
\end{align*}
for any cubes $Q$. Estimate \eqref{psi-1} can not be proved in this way, since $\psi\notin BMO(\mathbb R^d)$. The problem is at the unit sphere. Let ${\mathbf e}_1=(1,0,\cdots,0)$, consider the ball $B({\mathbf e}_1,\varepsilon)$, where $0<\varepsilon<\frac18$. Note that $B\left((1-\frac\varepsilon2){\mathbf e}_1,\frac\varepsilon2\right)\subseteq B({\mathbf e}_1,\varepsilon)\cap B(0,1)$, which can be represented as
\begin{align*}
B\left((1-\frac\varepsilon2){\mathbf e}_1,\frac\varepsilon2\right)=\bigg\{
\theta_2,\theta_3,\cdots,\theta_{d-2}\in[0,\pi],\theta_{d-1}\in[0,2\pi),
\theta_1\in[0,\arcsin{\frac{\varepsilon}{2-\varepsilon}}],\\r\in \left(r_{1}(\theta_1),r_{2}(\theta_1)\right)\bigg\},
\end{align*}
where
\begin{align*}
r_{1}(\theta_1)=\frac{(2-\varepsilon)\cos\theta_1-\sqrt{\varepsilon^2\cos^2\theta_1-4(1-\varepsilon)\sin^2\theta_1}}2
\end{align*}
and
\begin{align*}
r_{2}(\theta_1)=\frac{(2-\varepsilon)\cos\theta_1+\sqrt{\varepsilon^2\cos^2\theta_1-4(1-\varepsilon)\sin^2\theta_1}}2.
\end{align*}
Observe that
\begin{align*}
\frac1{|B({\mathbf e}_1,\varepsilon)|}\int_{B({\mathbf e}_1,\varepsilon)}\left|\psi(y)-\psi_{B({\mathbf e}_1,\varepsilon)}\right|dy&\ge \frac{|B({\mathbf e}_1,\varepsilon)\setminus B(0,1)|}{|B({\mathbf e}_1,\varepsilon)|}\psi_{B({\mathbf e}_1,\varepsilon)}\ge\frac{\psi_{B({\mathbf e}_1,\varepsilon)}}2.
\end{align*}
Using the spherical coordinate formula, we have
\begin{align*}
\psi_{B({\mathbf e}_1,\varepsilon)}&=-\frac1{|B({\mathbf e}_1,\varepsilon)|}\int_{B({\mathbf e}_1,\varepsilon)\cap B(0,1)}\ln\big(1-|y|\big)dy-\frac1{|B({\mathbf e}_1,\varepsilon)|}\int_{B({\mathbf e}_1,\varepsilon)\cap B(0,1)}\ln\big(1+|y|\big)dy\\
&\ge-\frac1{|B({\mathbf e}_1,\varepsilon)|}\int_0^{2\pi}\cdots\int_0^{\arcsin{\frac{\varepsilon}{2-\varepsilon}}}\int_{r_1(\theta_1)}^{r_2(\theta_1)}\ln(1-r)r^{d-1}dr\sin^{d-2}\theta_1d\theta_1\cdots d\theta_{d-1}-\ln2\\
&\ge -\frac{\omega_{d-2}}{\varepsilon^d\omega_{d-1}}\frac{\ln\varepsilon}{2^{d-1}}\int_0^{\arcsin{\frac{\varepsilon}{2-\varepsilon}}}\sqrt{\varepsilon^2\cos^2\theta_1-4(1-\varepsilon)\sin^2\theta_1}\sin^{d-2}\theta_1d\theta_1-\ln2.
\end{align*}
Let $t=\frac{2-\varepsilon}{\varepsilon}$, then
\begin{align*}
\int_0^{\arcsin{\frac{\varepsilon}{2-\varepsilon}}}\sqrt{\varepsilon^2\cos^2\theta_1-4(1-\varepsilon)\sin^2\theta_1}\sin^{d-2}\theta_1d\theta_1&=\frac2{t+1}\int_0^{\arcsin{\frac1t}}\sqrt{1-t^2\sin^2\theta_1}\sin^{d-2}\theta_1d\theta_1.
\end{align*}
Note that $\varepsilon\rightarrow0^+$ if and only if $t\rightarrow+\infty$. Using Proposition \ref{fn} below, we deduce 
\begin{align*}
&\lim_{\varepsilon\rightarrow0^+}\bigg(\frac{2-\varepsilon}{\varepsilon}\bigg)^{d-1}\frac1{\varepsilon}\int_0^{\arcsin{\frac{\varepsilon}{2-\varepsilon}}}\sqrt{\varepsilon^2\cos^2\theta_1-4(1-\varepsilon)\sin^2\theta_1}\sin^{d-2}\theta_1d\theta_1\\
&=\lim_{t\rightarrow+\infty}t^{d-1}\int_0^{\arcsin{\frac1t}}\sqrt{1-t^2\sin^2\theta_1}\sin^{d-2}\theta_1d\theta_1\\
&=\frac{\sqrt{\pi}}{2d}\cdot\frac{\Gamma(\frac{d-1}2)}{\Gamma(\frac d2)}.
\end{align*}
It follows that
\begin{align*}
\psi_{B({\mathbf e}_1,\varepsilon)}\ge -\frac{\omega_{d-2}}{\omega_{d-1}}\frac{\ln\varepsilon}{(4-2\varepsilon)^{d-1}}\frac{\sqrt{\pi}}{2d}\cdot\frac{\Gamma(\frac{d-1}2)}{\Gamma(\frac d2)}-\ln2
\end{align*}
when $\varepsilon$ is small enough. Finally, we get 
\begin{align*}
\lim_{\varepsilon\rightarrow0^+}\frac1{|B({\mathbf e}_1,\varepsilon)|}\int_{B({\mathbf e}_1,\varepsilon)}\left|\psi(y)-\psi_{B({\mathbf e}_1,\varepsilon)}\right|dy=+\infty.
\end{align*}
\end{remark}
\begin{proposition}\label{fn}
For $n\in\mathbb N$, let 
\begin{align*}
f_n(t)=t^{n}\int_0^{\arcsin{\frac1{t}}}\sqrt{1-t^2\sin^2\theta}\sin^{n-1}\theta d\theta, \ \ t\in[1,+\infty).
\end{align*}
Then $f_n$ is decreasing on $[1,+\infty)$. Moreover, 
\begin{align*}
\lim_{t\rightarrow+\infty}f_n(t)=\frac{\sqrt{\pi}}{2(n+1)}\cdot\frac{\Gamma(\frac n2)}{\Gamma(\frac {n+1}2)}.
\end{align*}
\end{proposition}
\begin{proof}
For any $t\ge1$, by changing of variables, we write
\begin{align*}
f_n(t)&=t^{n}\int_0^{\frac1{t}}\sqrt{1-t^2u^2}\frac{u^{n-1}du}{\sqrt{1-u^2}}=\int_0^{1}\sqrt{\frac{1-v^2}{1-(\frac vt)^2}}v^{n-1}dv.
\end{align*}
Set
\begin{align*}
\phi_n(s)=\int_0^{1}\sqrt{\frac{1-v^2}{1-(sv)^2}}v^{n-1}dv, \ \ s\in[0,1].
\end{align*}
Observe that $\phi_n(s)$ is increasing and continuous on $[0,1]$. Therefore, $f_n(t)$ is decreasing on $[1,+\infty)$ and 
\begin{align*}
\lim_{t\rightarrow+\infty}f_n(t)=\lim_{s\rightarrow0^+}\phi_n(s)&=\int_0^{1}\sqrt{1-v^2}v^{n-1}dv=\frac1{n+1}\int_0^{\frac{\pi}2}\sin^{n-1}\theta d\theta\\
&=\frac{\sqrt{\pi}}{2(n+1)}\cdot\frac{\Gamma(\frac n2)}{\Gamma(\frac {n+1}2)}.
\end{align*}
\end{proof}
\par
\subsection{Proofs of Theorem \ref{notw11}}

 It was proved in \cite[Lemma 1.6]{CP95} that
\begin{align}\label{M2llogl}
\left|\left\{x\in\mathbb R^d:M^2 f(x)>\lambda\right\}\right|\le C4^d\int_{\mathbb R^d}\Phi\left(\frac{|f(x)|}\lambda\right)dx,\ \ \lambda>0,
\end{align}
where $\Phi(s)=s\left(1+\log^+s\right)$. Note that $a^d\nu_{d-1}$ approaches $0$ as $d$ goes to infinity for any $a>0$. For $0<C<1$, convexity gives $\Phi(Cs)\le C\Phi(s)$ for all $s\ge0$. Proposition \ref{M2} yields the “$L\log L$” type estimate \eqref{epe} in Theorem \ref{notw11}.
\par
It is turn to show $M_\psi$ is not of weak type $(1,1)$. For $d\ge2$, we take $f(y)=\chi_{B(0,1)}(y)$ and write
\begin{align*}
\psi_t\ast f(x)=-\int_{\mathbb R^d}\ln(1-|y|^2)\chi_{B(\frac xt,\frac1t)\cap B(0,1)}(y)dy.
\end{align*}
Obviously, $\psi_t\ast f$ and $M_\psi f(x)$ are radial. We estimate $M_\psi f(x)$ for $|x|\ge 4$. Without loss of generality, we assume that $x$ is on the $x_1$-axis, namely, $x=|x|{\mathbf e}_1$. 
\par
For any $x\notin B(0,4)$, take $t=|x|+1$, we have $B(\frac xt,\frac1t)\cap B(0,1)=B(\frac xt,\frac1t)$, which can be represented as
\begin{align*}
B\left(\frac x{t},\frac1{t}\right)=\left\{
\theta_2,\cdots,\theta_{d-2}\in[0,\pi],\theta_{d-1}\in[0,2\pi),
\theta_1\in\left[0,\arcsin{\frac1{|x|}}\right], r\in\left(r_{1}(\theta_1),r_{2}(\theta_1)\right)\right\},
\end{align*}
where 
\begin{align*}
r_{1}(\theta_1)=\frac{|x|\cos\theta_1-\sqrt{1-|x|^2\sin^2\theta_1}}t
\end{align*}
and
\begin{align*}
r_{2}(\theta_1)=\frac{|x|\cos\theta_1+\sqrt{1-|x|^2\sin^2\theta_1}}t.
\end{align*}
Therefore, the change of variables yields
\begin{align*}
\psi_t\ast f(x)&=-\int_{B(\frac x{t},\frac1{t})}\ln\big(1-|y|^2\big)dy=-\int_{B(\frac x{t},\frac1{t})}\ln\big(1-|y|\big)dy-\int_{B(\frac x{t},\frac1{t})}\ln\big(1+|y|\big)dy\\
&\ge-\int_0^{2\pi}\cdots\int_0^{\arcsin{\frac1{|x|}}}\int_{r_1(\theta_1)}^{r_2(\theta_1)}\ln(1-r)r^{d-1}dr\sin^{d-2}\theta_1d\theta_1\cdots d\theta_{d-1}-\frac{\ln2\omega_{d-1}}{(|x|+1)^d}\\
&\ge \frac{\omega_{d-2}}{2^{d}}\frac{\ln(|x|+1)}{|x|+1}\int_0^{\arcsin{\frac1{|x|}}}\sqrt{1-|x|^2\sin^2\theta_1}\sin^{d-2}\theta_1d\theta_1-\frac{\ln2\omega_{d-1}}{(|x|+1)^d}\\
&\ge C_d\frac{\ln(|x|+1)}{(|x|+1)^d}-\frac{\ln2\omega_{d-1}}{(|x|+1)^d},
\end{align*}
where we have used Proposition \ref{fn}. It follows that there are two positive constants $C_0\ge4$ and $A$ such that 
\begin{align*}
M_\psi f(x)\ge\frac{A\ln|x|^d}{|x|^{d}},\ \ |x|\ge C_0.
\end{align*}
\par
For $0<\lambda\le\frac1e$, $\frac1s\ln{s}=\lambda$ has a solution $s_\lambda$ in $(e,+\infty)$. It is known that $s_\lambda=-\frac{W_{-1}(-\lambda)}\lambda$ and is a continuous, decreasing function on $(0,\frac1e)$, where $W_{-1}$ is the non-principal branch of Lambert $W$-function.
It is proved in \cite[Theorem 3.5]{AS20} that
\begin{align}\label{est of W}
\frac{\sqrt3}{12} r(1+e^{-r})+1<-W_{-1}(-z)<\frac12r (1+e^{-r})+1, \ \ r>0,
\end{align}
where $z(r)=\frac1e\frac{re^{\frac r2}}{e^r-1}$. Note that $z(r)$ is continuous, decreasing on $(0,+\infty)$ and
 \begin{align*}
\lim_{r\rightarrow 0^+}z(r)=\frac1e,\ \ \lim_{r\rightarrow +\infty}z(r)=0.
\end{align*}
\par
For any $\lambda\in(0,\frac Ae)$, there is only one $r_{\frac{\lambda}{A}}$ such that $z(r_{\frac{\lambda}{A}})=\frac{\lambda}{A}$. Using \eqref{est of W}, we deduce that
\begin{align*}
\lambda\left|\left\{x\in\mathbb R^d:M_\psi f(x)>\lambda\right\}\right|&\ge\lambda\left|\left\{x\notin B(0,C_0):\frac{\ln|x|^d}{|x|^d}>\frac{\lambda}{A}\right\}\right|\\
&\ge\lambda\left|\left\{x\in \mathbb R^d:C_0\le|x|<\left({s_{\frac{\lambda}{A}}}\right)^{1/d}\right\}\right|\\
&=\lambda\omega_{d-1}\left( s_{\frac{\lambda}{A}}-C_0^d\right)=\omega_{d-1}\left[-{AW_{-1}\big(-\frac{\lambda}{A}\big)}-C_0^d\lambda \right]\\
&\ge \omega_{d-1}\left[\frac{A\sqrt3}{12} r_{\frac{\lambda}{A}}\left(1+e^{-r_{\frac{\lambda}{A}}}\right)+A-C_0^d\lambda \right].
\end{align*}
Note that $r_{\frac{\lambda}{A}}\rightarrow+\infty$ as $\lambda$ tends to $0^+$. It follows that
\begin{align*}
\sup_{\lambda>0}\lambda\left|\left\{x\in\mathbb R^d:M_\psi f(x)>\lambda\right\}\right|=+\infty.
\end{align*}
This proves $M_\psi$ is not of weak type $(1,1)$.

\begin{remark}\label{another proof}
(1). The negative result in Theorem \ref{notw11} can be interpreted in another way. Let $K\in L^1(\mathbb R^d)$ and $\{t_j\}$ be a nonincreasing sequence of positive number tending to $0$ which is non-lacunary, i.e. $t_{j+1}/t_j\rightarrow1$ as $j\rightarrow\infty$. The non-lacunary maximal operator associated to $K$ is given by
\begin{align*}
K^\ast f(x)=\sup_{j\in\mathbb Z}|K_{t_j}\ast f(x)|.
\end{align*}
Carrillo and Guzm\'{a}n  proved that if  $K^\ast$
is of weak type $(1,1)$, then there exists a constant $C_d>0$ such that for each $\lambda>0$ we have
\begin{align}\label{HK}
\left|\left\{x'\in\mathbb S^{d-1}:H_K(x')>\lambda\right\}\right|\le\frac {C_d}\lambda,
\end{align}
where 
\begin{align*}
H_K(x')=\esssup_{r>0}r^d|K(rx')|,\ \ x'\in\mathbb S^{d-1}. 
\end{align*}
see \cite[Theorem 7]{CG82}. For the kernel $\psi\in L^1(\mathbb R^d)$, $H_\psi(x')=+\infty$ for any $x'\in\mathbb S^{d-1}$. Obviously, \eqref{HK} does not hold, it follows that any non-lacunary maximal operator associated with $\psi$ is not of weak type $(1,1)$. Moreover, $M_\psi$ is not of weak type $(1,1)$. 
\par
(2). Note that $H^1(\mathbb R^d)\subsetneqq L^1(\mathbb R^d)$.Whether a weaker endpoint inequality holds, namely that $M_\psi$ maps $H^1(\mathbb R^d)$ to $L^{1,\infty}(\mathbb R^d)$? Let $B^-(0,1)=\{x\in B(0,1):x_1\le0\}$ and $B^+(0,1)=\{x\in B(0,1):x_1>0\}$, $f(y)=\chi_{B^-(0,1)}(y)-\chi_{B^+(0,1)}(y)$, then $f$ is an $L^\infty$-atom and $f\in H^1(\mathbb R^d)$. A similar calculation
is valid, we also have 
\begin{align*}
\sup_{\lambda>0}\lambda\left|\left\{x\in\mathbb R^d:M_\psi f(x)>\lambda\right\}\right|=+\infty.
\end{align*}
Hence, $M_\psi$ does not map $H^1(\mathbb R^d)$ to $L^{1,\infty}(\mathbb R^d)$.
\end{remark}

\subsection{Proof of Theorem \ref{dyad w11}}
Set $\varphi(y)=-\ln(1-|y|)\chi_{B(0,1)}(y)$ and define
$$
M_\varphi^Df(x)=\sup_{k\in\mathbb Z}\left|\varphi_{2^k}\ast f(x)\right|.
$$
Now observe that
\begin{align*}
M^D_\psi f(x)\le M_\varphi^Df(x)\le M_\psi^Df(x)+\ln2\cdot Mf(x)
\end{align*}
It suffices to show that $M_\varphi^D$ is of weak type $(1,1)$.
\par
Theorem \ref{ppe} implies that $M_\varphi^D$ is bounded on $L^p(\mathbb R^d)$ for $1<p\le\infty$. We apply vector-valued estimates for the sublinear operator (see \cite[Corollary 5.6.5]{G14}), it suffices to verify the following regularity condition
\begin{align*}
\sup_{y\neq0}\int_{|x|\ge2|y|}\sup_{j\in\mathbb Z}\left|\varphi_{2^j}(x-y)-\varphi_{2^j}(x)\right|dx<\infty.
\end{align*}
\par
Without loss of generality, we assume that $y=(|y|,0,\cdots,0)$ and $|y|\neq0$. Then, there is a $k_0\in\mathbb Z$ such that  $|y|\in[2^{k_0-1},2^{k_0})$. For $|x|\ge2|y|$, we observe that $x_1>\frac{|y|}2$ if and only if $|x|>|x-y|$.
\par
Set $R_{k_0}=\{x:2|y|\le |x|<2^{k_0+1}\}$, $R_k=\{x:2^k\le |x|<2^{k+1}\}$ for $k\ge k_0+1$. We write
\begin{align*}
\int_{|x|\ge2|y|}\sup_{j\in\mathbb Z}\left|\varphi_{2^j}(x-y)-\varphi_{2^j}(x)\right|dx=\sum_{k\ge k_0}\int_{R_k}\sup_{j\in\mathbb Z}\left|\varphi_{2^j}(x-y)-\varphi_{2^j}(x)\right|dx.
\end{align*}
\par
To estimate each summand, we divide each $R_k$ into four parts: \\
$$R_{k,1}:=\left\{x\in R_k:x_1>\frac{|y|}2,x\notin B(y,2^{k})\right\},$$
$$R_{k,2}:=\left\{x\in R_k:x_1>\frac{|y|}2,x\in B(y,2^{k})\right\},$$  
$$R_{k,3}:=\left\{x\in R_k:x_1\le\frac{|y|}2,x\in B(y,2^{k+1})\right\},$$
$$R_{k,4}:=\left\{x\in R_k:x_1\le\frac{|y|}2,x\notin B(y,2^{k+1})\right\}.$$
\begin{figure}[h]
    \centering
    \begin{subfigure}[b]{0.49\textwidth}
    \includegraphics[width=\textwidth]{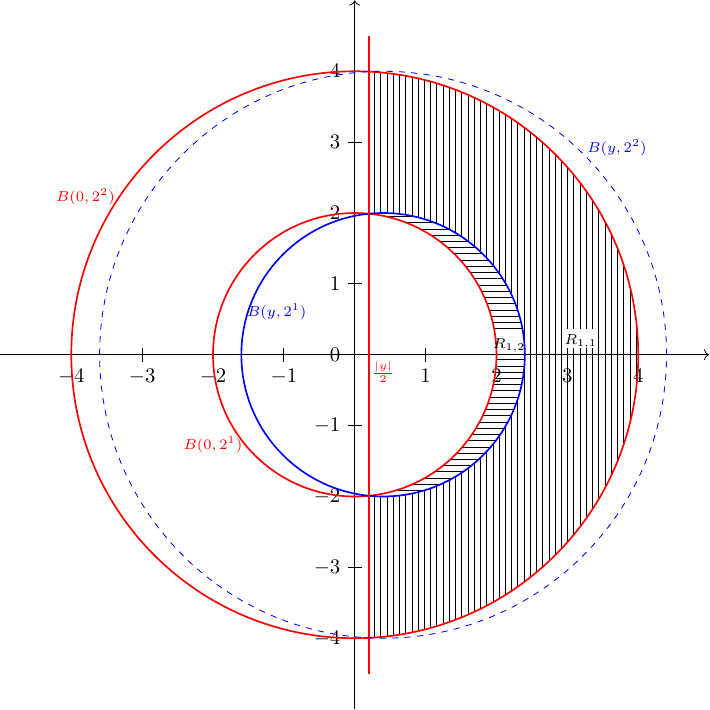} 
    \caption{$R_{1,1},R_{1,2}$}
    \label{fig:my_image} 
\end{subfigure}
\begin{subfigure}[b]{0.49\textwidth}
    \includegraphics[width=\textwidth]{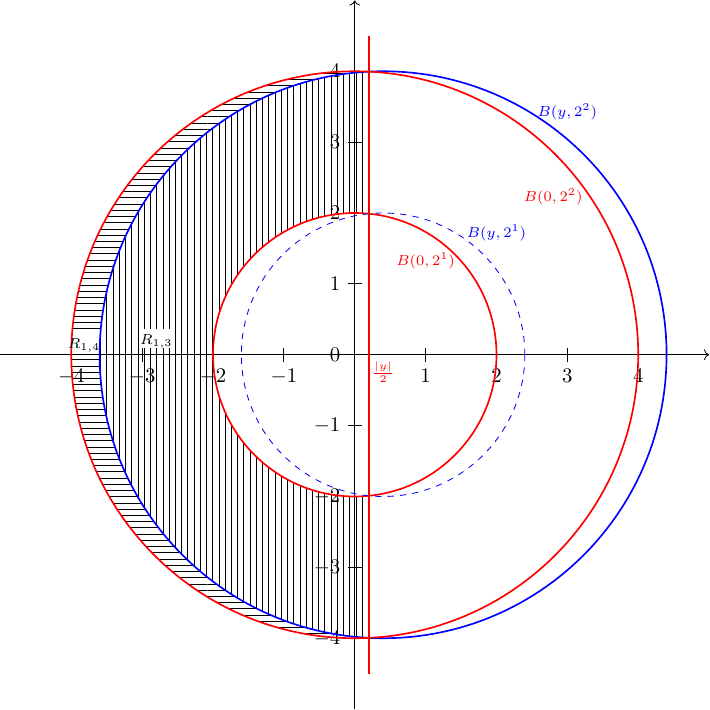} 
    \caption{$R_{1,3},R_{1,4}$}
    \label{fig:my_image} 
\end{subfigure}
\end{figure}
\par
For each summand, we write it as the sum of four terms
\begin{align*}
\int_{R_k}\sup_{j\in\mathbb Z}\left|\varphi_{2^j}(x-y)-\varphi_{2^j}(x)\right|dx&=\sum_{i=1}^4\int_{R_{k,i}}\sup_{j\in\mathbb Z}\left|\varphi_{2^j}(x-y)-\varphi_{2^j}(x)\right|dx.
\end{align*}
We work with summands for $k\ge k_0+1$, the summand with $k=k_0$ can be treated similarly.
\par
If $x\in R_{k,1}$, then $2^{k}\le|x-y|<|x|<2^{k+1}$.  For $j\le k$, $\varphi_{2^j}(x-y)=\varphi_{2^j}(x)=0$. For $j\ge k+1$, we get $\varphi_{2^j}(x)>\varphi_{2^j}(x-y)>0$, and
\begin{align*}
\left|\varphi_{2^j}(x)-\varphi_{2^j}(x-y)\right|=\varphi_{2^j}(x)-\varphi_{2^j}(x-y)=\frac1{2^{jd}}\ln\left(1+\frac{|x|-|x-y|}{2^j-|x|}\right).
\end{align*}
Observe that $\varphi_{2^j}(x)-\varphi_{2^j}(x-y)$ is decreasing with respect to $j$ when $j\ge k+1$. Thus,
\begin{align*}
\sup_{j\in\mathbb Z}\big|\varphi_{2^j}(x-y)-\varphi_{2^j}(x)\big|=\frac1{2^{(k+1)d}}\ln\left(1+\frac{|x|-|x-y|}{2^{k+1}-|x|}\right).
\end{align*}
Note that $|x|-|x-y|\le |y|$. Therefore, using the change of variables we deduce that
\begin{align*}
\int_{R_{k,1}}\sup_{j\in\mathbb Z}\left|\varphi_{2^j}(x-y)-\varphi_{2^j}(x)\right|dx&\le\frac1{2^{(k+1)d}}\int_{2^k\le|x|<2^{k+1}}\ln\left(1+\frac{|y|}{2^{k+1}-|x|}\right)dx\\
&\le\frac{\omega_{d-1}|y|}{2^{k+1}}\int_{\frac {|y|}{2^k}}^\infty\frac{\ln(1+s)}{s^2}ds\\
&\le 3\omega_{d-1}\frac{k-k_0}{2^{k-k_0}}.
\end{align*}
\par
If $x\in R_{k,2}$, then $2^k-|y|\le|x-y|<2^k$. For $j<k$, $\frac{|x|}{2^j}>\frac{|x-y|}{2^j}\ge\frac{|x|-|y|}{2^j}\ge \frac{|x|}{2^{j+1}}\ge1$, so $\varphi_{2^j}(x-y)=\varphi_{2^j}(x)=0$. 
For $j=k$, $\frac{|x-y|}{2^j}<1$ and $\frac{|x|}{2^j}\ge1$, so $\varphi_{2^j}(x-y)>0$ and $\varphi_{2^j}(x)=0$. Then,
\begin{align*}
\left|\varphi_{2^j}(x)-\varphi_{2^j}(x-y)\right|=\varphi_{2^k}(x-y).
\end{align*}
For $j\ge k+1$, $\frac{|x-y|}{2^j}<\frac{|x|}{2^j}<1$, so $\varphi_{2^j}(x)>\varphi_{2^j}(x-y)>0$. We deduce that
\begin{align*}
\left|\varphi_{2^j}(x)-\varphi_{2^j}(x-y)\right|=\varphi_{2^j}(x)-\varphi_{2^j}(x-y)\le \varphi_{2^{k+1}}(x)-\varphi_{2^{k+1}}(x-y).
\end{align*}
\par
Moreover, for $x\in R_{k,2}$, we have $|x-y|<2^k\le|x|$. Observe that  $|x-y|>|y|$ gives 
\begin{align*}
\frac{|x-y|}{2^{k}}\left(\frac32-\frac{|x-y|}{2^{k}}\right)>\frac{|y|}{2^{k+1}}. 
\end{align*}
It follows that 
\begin{align*}
\frac{|x-y|}{2^{k-1}}\ge\frac{|x-y|+|y|}{2^{k+1}}+\frac{|x-y|^2}{2^{2k}}\ge \frac{|x|}{2^{k+1}}+\frac{|x-y|^2}{2^{2k}},
\end{align*}
and
\begin{align*}
1- \frac{|x|}{2^{k+1}}\ge \left(1-\frac{|x-y|}{2^{k}}\right)^2\ge\left(1-\frac{|x-y|}{2^{k}}\right)^{2^d}.
\end{align*}
Thus we prove that
\begin{align*}
\varphi_{2^{k+1}}(x)-\varphi_{2^{k+1}}(x-y)\le \varphi_{2^{k+1}}(x)&=-\frac1{2^{(k+1)d}}\ln\left(1- \frac{|x|}{2^{k+1}}\right)\\
&\le -\frac1{2^{kd}}\ln\left(1-\frac{|x-y|}{2^{k}}\right)=\varphi_{2^{k}}(x-y),
\end{align*}
which yields
\begin{align*}
\sup_{j\in\mathbb Z}\left|\varphi_{2^j}(x-y)-\varphi_{2^j}(x)\right|=-\frac1{2^{kd}}\ln\left(1-\frac{|x-y|}{2^{k}}\right).
\end{align*}
Therefore, using the change of variables we have
\begin{align*}
\int_{R_{k,2}}\sup_{j\in\mathbb Z}\left|\varphi_{2^j}(x-y)-\varphi_{2^j}(x)\right|dx&\le-\frac1{2^{kd}}\int_{2^k-|y|\le|x-y|<2^k}\ln\left(1-\frac{|x-y|}{2^{k}}\right)dx\\
&\le-\omega_{d-1}\int_{1-\frac{|y|}{2^k}}^1\ln(1-s)ds\le 3\omega_{d-1}\frac{k-k_0}{2^{k-k_0}}.
\end{align*}
\par
 If $x\in R_{k,3}$, then $2^k\le|x|<|x-y|<2^{k+1}$. For $j\le k$, since $\frac{|x-y|}{2^j}\ge\frac{|x|}{2^j}\ge1$, this implies that $\varphi_{2^j}(x-y)=\varphi_{2^j}(x)=0$. 
For $j\ge k+1$, $\frac{|x|}{2^j}<\frac{|x-y|}{2^j}<1$, $\varphi_{2^j}(x-y)>\varphi_{2^j}(x)>0$. As in the case $x\in R_{k,1}$ we have
\begin{align*}
\sup_{j\in\mathbb Z}\left|\varphi_{2^j}(x)-\varphi_{2^j}(x-y)\right|&=\varphi_{2^{k+1}}(x-y)-\varphi_{2^{k+1}}(x)\\
&=\frac1{2^{(k+1)d}}\ln\left(1+\frac{|x-y|-|x|}{2^{k+1}-|x-y|}\right).
\end{align*}
Note that $|x-y|-|x|\le |y|$. Therefore, using the change of variables we deduce that
\begin{align*}
\int_{R_{k,3}}\sup_{j\in\mathbb Z}\left|\varphi_{2^j}(x-y)-\varphi_{2^j}(x)\right|dx&\le\frac1{2^{(k+1)d}}\int_{2^k\le|x-y|<2^{k+1}}\ln\left(1+\frac{|y|}{2^{k+1}-|x-y|}\right)dx\\
&\le\frac{\omega_{d-1}|y|}{2^{k+1}}\int_{\frac {|y|}{2^k}}^\infty\frac{\ln(1+s)}{s^2}ds\\
&\le 3\omega_{d-1}\frac{k-k_0}{2^{k-k_0}}.
\end{align*}
\par
 If $x\in R_{k,4}$, then $2^{k+1}-|y|\le|x|<2^{k+1}$. For $j\le k$, $\frac{|x-y|}{2^j}>\frac{|x|}{2^j}\ge\frac{|x-y|-|y|}{2^j}\ge \frac{|x-y|}{2^{j+1}}\ge1$, it follows that $\varphi_{2^j}(x-y)=\varphi_{2^j}(x)=0$. For $j= k+1$, $\frac{|x-y|}{2^j}\ge1>\frac{|x|}{2^{j}}$, so $\varphi_{2^j}(x)>0$, $\varphi_{2^j}(x-y)=0$ and
\begin{align*}
\left|\varphi_{2^j}(x)-\varphi_{2^j}(x-y)\right|=\varphi_{2^{k+1}}(x).
\end{align*}
For $j\ge k+2$, $\frac{|x|}{2^j}<\frac{|x-y|}{2^j}<\frac{|x|+|y|}{2^j}<1$, then $\varphi_{2^j}(x-y)>\varphi_{2^j}(x)>0$ and
\begin{align*}
\left|\varphi_{2^j}(x)-\varphi_{2^j}(x-y)\right|=\varphi_{2^j}(x-y)-\varphi_{2^j}(x)\le \varphi_{2^{k+2}}(x-y)-\varphi_{2^{k+2}}(x).
\end{align*}
\par
Moreover, for $x\in R_{k,4}$, we have $|x|<2^{k+1}<|x-y|$. Observe that $|x|>|y|$ gives that $\frac{|x|}{2^{k+1}}\big(\frac32-\frac{|x|}{2^{k+1}}\big)\ge\frac{|y|}{2^{k+2}}$. It follows that
\begin{align*}
\frac{|x|}{2^k}&\ge\frac{|x|+|y|}{2^{k+2}}+\frac{|x|^2}{2^{2(k+1)}}\ge \frac{|x-y|}{2^{k+2}}+\frac{|x|^2}{2^{2(k+1)}}
\end{align*}
and
\begin{align*}
1- \frac{|x-y|}{2^{k+2}}\ge \bigg(1-\frac{|x|}{2^{k+1}}\bigg)^2\ge\bigg(1-\frac{|x|}{2^{k+1}}\bigg)^{2^d}.
\end{align*}
Thus we proved that
\begin{align*}
\varphi_{2^{k+2}}(x-y)-\varphi_{2^{k+2}}(x)<\varphi_{2^{k+2}}(x-y)&=-\frac1{2^{(k+2)d}}\ln\left(1- \frac{|x-y|}{2^{k+2}}\right)\\
&\le -\frac1{2^{(k+1)d}}\ln\left(1-\frac{|x|}{2^{k+1}}\right)=\varphi_{2^{k+1}}(x),
\end{align*}
which yields
\begin{align*}
\sup_{j\in\mathbb Z}\left|\varphi_{2^j}(x-y)-\varphi_{2^j}(x)\right|=\varphi_{2^{k+1}}(x).
\end{align*}
Therefore, using the change of variables we deduce that
\begin{align*}
\int_{R_{k,4}}\sup_{j\in\mathbb Z}\left|\varphi_{2^j}(x-y)-\varphi_{2^j}(x)\right|dx&\le-\frac1{2^{(k+1)d}}\int_{2^{k+1}-|y|\le|x|<2^{k+1}}\ln\left(1-\frac{|x|}{2^{k+1}}\right)dx\\
&\le-\omega_{d-1}\int_{1-\frac{|y|}{2^{k+1}}}^1\ln(1-s)ds\le 3\omega_{d-1}\frac{k-k_0}{2^{k-k_0}}.
\end{align*}
\par
Finally, combining above four estimates, we obtain
\begin{align*}
\int_{|x|\ge2|y|}\sup_{j\in\mathbb Z}\left|\varphi_{2^j}(x-y)-\varphi_{2^j}(x)\right|dx&=\sum_{k\ge k_0}\int_{R_k}\sup_{j\in\mathbb Z}\left|\varphi_{2^j}(x-y)-\varphi_{2^j}(x)\right|dx\\
&\le 12\omega_{d-1}\sum_{k\ge k_0+1}\frac{k-k_0}{2^{k-k_0}}\le 24\omega_{d-1}.
\end{align*}

\section{Dimension-free estimates for $M_\psi$} For $d\ge2$ and $1<k\le d$,
we denote by $d\sigma$ the normalized Lebesgue measure on the sphere $\mathbb S^{k-1}$. For $f\in L^p(\mathbb R^k)$, $1\le p\le\infty$, we define the spherical maximal function in $\mathbb R^k$
\begin{align*}
\mathscr M_k(f)(x)=\sup_{t>0}\left|\int_{\mathbb S^{k-1}}f(x-ty')d\sigma(y')\right|.
\end{align*}
$\mathscr M_k$ is bounded in $L^p(\mathbb R^k)$ for some $p$ and $k$, with a bound dependent on the dimension $k$, according to the following theorem due to Stein \cite[Theorem 1]{S76}.
\begin{lemma}\label{bounds for sph}
Let $k\ge3$. For each $\frac{k}{k-1}<p\le\infty$, there is a constant $C(k,p)$ such that 
\begin{align*}
\|\mathscr M_k(f)\|_{L^p(\mathbb R^k)}\le C(k,p)\|f\|_{L^p(\mathbb R^k)}
\end{align*}
holds for all $f$ in $L^p(\mathbb R^k)$.
\end{lemma}
We now consider $\mathbb R^d$ with $d\ge3$. For $1<k<d$, we write $\mathbb R^d$ as $\mathbb R^k\times \mathbb R^{d-k}$ and denote an $y\in\mathbb R^d$ as a pair $y=(\widetilde{y_1},\widetilde{y_2})$ with $\widetilde{y_1}\in\mathbb R^k$, $\widetilde{y_2}\in\mathbb R^{d-k}$. For each $\tau$ in the orthogonal group $\mathcal O(d)$ we define $M_k^\tau$ as
\begin{align*}
M_k^\tau f(x)=\sup_{t>0}\frac{(-1)}{t^d}\int_{\mathbb R^k}\ln\left(1-\left(\frac{|\widetilde{y_1}|}t\right)^2\right)\left|f\left(x-\tau(\widetilde{y_1},0)\right)\right||\widetilde{y_1}|^{d-k}d\widetilde{y_1},\ \ x\in\mathbb R^d.
\end{align*} 
Given $\tau\in\mathcal O(d)$, we set $f_\tau(x)=f(\tau x)$ for each $x\in\mathbb R^d$. The mapping $S_\tau:f\mapsto f_\tau$ is an isometry of $L^p(\mathbb R^d)$. It follows that
$S_\tau M_k^\tau=M_k^{Id}S_\tau$ and
\begin{align}\label{iso}
\|M_k^{\tau} f\|_{L^p(\mathbb R^d)}=\|M_k^{Id} f\|_{L^p(\mathbb R^d)},
\end{align}
where $Id$ is the identity rotation and
\begin{align*}
M_k^{Id} f(x)=\sup_{t>0}\frac{1}{t^d}\int_{\mathbb R^k}\ln\left(1-\left(\frac{|\widetilde{y_1}|}t\right)^2\right)\left|f(\widetilde{x_1}-\widetilde{y_1},\widetilde{x_2})\right||\widetilde{y_1}|^{d-k}d\widetilde{y_1}.
\end{align*} 
Using polar coordinates we observe that
\begin{align*}
&\frac{(-1)}{t^d}\int_{\mathbb R^k}\ln\left(1-\left(\frac{|\widetilde{y_1}|}t\right)^2\right)\left|f(\widetilde{x_1}-\widetilde{y_1},\widetilde{x_2})\right||\widetilde{y_1}|^{d-k}d\widetilde{y_1}\\
=&\frac{(-1)\omega_{k-1}}{t^d}\int_0^t\ln\left(1-\left(\frac{r}t\right)^2\right)r^{d-1}\int_{\mathbb S^{k-1}}\left|f\left(\widetilde{x_1}-r\widetilde{y_1}',\widetilde{x_2}\right)\right|d\sigma(\widetilde{y_1}')dr\\
\le &\frac{\omega_{k-1}}{\omega_{d-1}}\|\psi\|_{L^1(\mathbb R^d)}\mathscr M_{k}(f_{\widetilde{x_2}})(\widetilde{x_1}),
\end{align*} 
where $f_{\widetilde{x_2}}(\widetilde{y_1})=f(\widetilde{y_1},\widetilde{x_2})$.
Thus, we have
\begin{align}\label{conbysph}
M_k^{Id} f(x)\le \frac{\omega_{k-1}}{\omega_{d-1}}\|\psi\|_{L^1(\mathbb R^d)}\mathscr M_{k}(f_{\widetilde{x_2}})(\widetilde{x_1}).
\end{align} 
\eqref{iso},\eqref{conbysph} and Lemma \ref{bounds for sph} implies
\begin{align}\label{Mktaupp}
\nonumber\|M_k^{\tau} f\|^p_{L^p(\mathbb R^d)}&\le\left(\frac{\omega_{k-1}}{\omega_{d-1}}\right)^p\|\psi\|^p_{L^1(\mathbb R^d)}\int_{\mathbb R^{d-k}}\int_{\mathbb R^k}\left|\mathscr M_{k}(f_{\widetilde{x_2}})(\widetilde{x_1})\right|^pd\widetilde{x_1}d\widetilde{x_2}\\
\nonumber&\le C(k,p)^p \left(\frac{\omega_{k-1}}{\omega_{d-1}}\right)^p\|\psi\|^p_{L^1(\mathbb R^d)}\int_{\mathbb R^{d-k}}\int_{\mathbb R^k}\left|f_{\widetilde{x_2}}(\widetilde{x_1})\right|^pd\widetilde{x_1}d\widetilde{x_2}\\
&=C(k,p)^p\left(\frac{\omega_{k-1}}{\omega_{d-1}}\right)^p\|\psi\|^p_{L^1(\mathbb R^d)}\|f\|^p_{L^p(\mathbb R^d)},
\end{align}
where $k\ge3$ and $p>\frac k{k-1}$.
\begin{proposition}\label{MpsicontrbyMkt}
For every locally integrable function $f$ on $\mathbb R^d$ and $1<k<d$, one has the following pointwise inequality
\begin{align*}
M_\psi f(x)\le \frac{\omega_{d-1}}{\omega_{k-1}}\int_{\mathcal O(d)}M_k^\tau f(x)d\tau, \ \ x\in\mathbb R^d,
\end{align*}
where $d\tau$ denotes the normalized Haar measure on $\mathcal O(d)$.
\end{proposition}
\begin{proof}
We assume that $f$ has the form $f(y)=f_0(|y|)f_1(y')$, where $y'\in\mathbb S^{d-1}$ and $y=|y|y'$, since linear combinations of such functions are dense. Write $\widetilde{y_1}=|\widetilde{y_1}|\widetilde{y_1}'$, where $\widetilde{y_1}'\in\mathbb S^{k-1}$. Then $f\left(\tau(\widetilde{y_1},0)\right)=f_0(|\widetilde{y_1}|)f_1\left(\tau(\widetilde{y_1}',0)\right)$. By taking polar coordinates, we see that
\begin{align*}
&\frac{(-1)}{t^d}\int_{\mathcal O(d)}\int_{\mathbb R^k}\ln\left(1-\left(\frac{|\widetilde{y_1}|}t\right)^2\right)f\left(\tau(\widetilde{y_1},0)\right)|\widetilde{y_1}|^{d-k}d\widetilde{y_1}d\tau\\
&=\frac{(-1)}{t^d}\int_0^tf_0(r)\ln\left(1-\left(\frac{r}t\right)^2\right)r^{d-1}dr\cdot\omega_{k-1}\cdot\int_{\mathcal O(d)}\int_{\mathbb S^{k-1}}f_1\left(\tau(\widetilde{y_1}',0)\right)d\sigma(\widetilde{y_1}')d\tau
\end{align*} 
and similarly
\begin{align*}
\int_{\mathbb R^d}\psi_t(y) f(y)dy=\frac{(-1)}{t^d}\int_0^tf_0(r)\ln\left(1-\left(\frac{r}t\right)^2\right)r^{d-1}dr\cdot\omega_{d-1}\cdot\int_{\mathbb S^{d-1}}f_1(y')d\sigma(y').
\end{align*}
It was proved that \cite[(4.2)]{SS83} 
\begin{align*}
\int_{\mathbb S^{d-1}}f_1(y')d\sigma(y')=\int_{\mathcal O(d)}\int_{\mathbb S^{k-1}}f_1\left(\tau(\widetilde{y_1}',0)\right)d\sigma(\widetilde{y_1}')d\tau.
\end{align*}
It follows that
\begin{align*}
\int_{\mathbb R^d}\psi_t(y) f(x-y)dy=\frac{\omega_{d-1}}{\omega_{k-1}}\int_{\mathcal O(d)}\frac{(-1)}{t^d}\int_{\mathbb R^k}\ln\left(1-\left(\frac{|\widetilde{y_1}|}t\right)^2\right)f\left(x-\tau(\widetilde{y_1},0)\right)|\widetilde{y_1}|^{d-k}d\widetilde{y_1}d\tau.
\end{align*}
This completes the proof.
\end{proof}
\begin{proof}[Proof of Theorem \ref{dfppe}]
Let $1<p\le+\infty$. When $d\le \max\{4,\frac p{p-1}\}+1$, there is a positive constant $C_p$ independent of the dimension $d$ such that
\begin{align*}
4^d\le -C_p\int_0^1\ln(1-r^2)r^{d-1}dr.
\end{align*}
 Theorem \ref{ppe} implies
\begin{align*}
\|M_\psi f\|_{L^p(\mathbb R^d)}\le -C_p\omega_{d-1}\int_0^1\ln(1-r^2)r^{d-1}dr\cdot\|f\|_{L^p(\mathbb R^d)}\le C_p\|\psi\|_{L^1(\mathbb R^d)}\|f\|_{L^p(\mathbb R^d)}.
\end{align*}
Therefore, we assume that $d>\max\{\frac p{p-1},4\}+1$. We take $k=\lceil \max\{\frac p{p-1},3\}\rceil$ and write $d=k+(d-k)$. Proposition \ref{MpsicontrbyMkt} and \eqref{Mktaupp} yield
\begin{align*}
\|M_\psi f\|_{L^p(\mathbb R^d)}\le \frac{\omega_{d-1}}{\omega_{k-1}}\int_{\mathcal O(d)}\|M_k^\tau f\|_{L^p(\mathbb R^d)}d\tau\le C(k,p)\|\psi\|_{L^1(\mathbb R^d)}\|f\|_{L^p(\mathbb R^d)},
\end{align*}
where $C(k,p)$ is a function of $p$ alone. This completes the proof of Theorem \ref{dfppe}.
\end{proof}

\section{Applications to dimension-free estimates for $R_j^\ast$}
For $j=1,2,\cdots,d$, Cotlar's pointwise inequality implies that
\begin{align*}
R_j^\ast f(x)\le M(R_jf)(x)+C_d M(f)(x),
\end{align*}
and
\begin{align}\label{LpRj}
\|R_j^\ast f\|_{L^p(\mathbb R^d)}\le C_{p,d}\|f\|_{L^p(\mathbb R^d)},\ \ 1<p<\infty.
\end{align}
\par
 In \cite{MV06}, Mateu and Verdera showed that
  \begin{align*}
 K^j(x)=R_j(h)(x),\ \ x\in\mathbb R^d,
 \end{align*}
 where
$$h(x)=\sum_{j=1}^dR_j(K^j)(x).$$ 
 They claimed that $h\in L^1(\mathbb R^d)$, but $h$ has not an integrable radial decreasing majorant. By estimating the size of $h$, they proved that for each $s>1$ there exists a positive constant $C_s$ such that
\begin{align}\label{Ms}
R_j^\ast f(x)\le C_sM\left(|R_jf|^s\right)^{\frac1s}(x)\ \ x\in\mathbb R^d,
\end{align} 
which gives
\begin{align*}
\|R_j^\ast f\|_{L^p(\mathbb R^d)}\le C_{p,d}\|R_jf\|_{L^p(\mathbb R^d)},\ \ 1<p<\infty.
\end{align*}
This improves \eqref{LpRj}. Note that the range of $s$ in \eqref{Ms} is sharp, i.e., above inequality does not hold when $s=1$, see \cite{MOPV10}. 
\par
To prove the following dimension-free $L^2$ estimate
\begin{align}\label{dfRj1}
\|R_j^\ast f\|_{L^2(\mathbb R^d)}\le 2\cdot 10^8\|R_jf\|_{L^2(\mathbb R^d)},\ \  j=1,2,\cdots,d,
\end{align}
 Kucharski and Wr\'{o}bel \cite{KW23} considered a function
$m:[0,+\infty)\rightarrow\mathbb C$ given by
 \begin{align*}
 m(s)=\frac{2^{\frac d2}\Gamma(\frac{d+1}2)}{\sqrt\pi}\int_{2\pi s}^\infty r^{-\frac d2}J_{\frac d2}(r)dr,
 \end{align*}
 and defined the maximal operator
 \begin{align*}
 M^\ast f(x)=\sup_{t>0}\left|M_tf(x)\right|
 \end{align*}
 by setting 
 \begin{align*}
 \widehat{M_tf}(\xi)=m(t|\xi|)\hat{f}(\xi), \ t>0, \ f\in L^2(\mathbb R^d).
 \end{align*}
Moreover, $\hat{h}(\xi)=m(|\xi|)$ and
 \begin{align*}
M_tf(x)=h_t\ast f(x), \ \ f\in L^2(\mathbb R^d).
 \end{align*}
It follows that
  \begin{align}\label{mRjM}
 R^\ast_jf(x)=\sup_{t>0}\left|K^j_t\ast f(x)\right|=\sup_{t>0}\left|h_t\ast (R_jf)(x)\right|=M^\ast(R_jf)(x).
 \end{align}
 To obtain \eqref{dfRj1}, they proved 
 \begin{align}\label{dfMast}
\|M^\ast f\|_{L^2(\mathbb R^d)}\le 2\cdot 10^8\|f\|_{L^2(\mathbb R^d)}.
\end{align}
\par
Recently, Liu, Melentijevi\'{c} and Zhu \cite{LMZ24} showed
 \begin{align}\label{dfMastLMZ}
\|M^\ast f\|_{L^p(\mathbb R^d)}\le \left(2+\frac1{\sqrt2}\right)^{\frac2p}\|f\|_{L^p(\mathbb R^d)},\ \ d\ge2,\ \ p\ge2,
\end{align}
by proving that
\begin{equation*}
   h(x) = \left\{
    \begin{aligned}
    & \frac{\Gamma(\frac{1+d}2)}{\pi^{\frac {d+3}2}}\sum_{n=0}^\infty\frac{\Gamma(n+\frac{1}2)\Gamma(n+\frac{1+d}2)}{\Gamma(n+1+\frac{d}2)n!}|x|^{2n},\ \ |x|<1,\\
    & \frac1{|x|^{1+d}}\frac{\Gamma(\frac{1+d}2)}{\pi^{\frac {d+3}2}}\sum_{n=0}^\infty\frac{\Gamma(n+\frac{1}2)\Gamma(n+\frac{1+d}2)}{\Gamma(n+1+\frac{d}2)n!}\bigg(\frac1{|x|}\bigg)^{2n},\ \ |x|>1,
    \end{aligned}
    \right.
\end{equation*}
$h(x)$ is radial, positive, $\|h\|_{L^1(\mathbb R^d)}=1$ and 
\begin{align*}
\lim_{|x|\rightarrow1^-}h(x)=\lim_{|x|\rightarrow1^+}h(x)=+\infty.
\end{align*}
\par
A natural question is whether dimension-free estimate \eqref{dfMastLMZ} holds for $1<p<2$? We give here a positive answer.
\begin{proposition}\label{dfMastpp}
Let $d\in\mathbb N$. For each $1<p\le\infty$, there is a constant $C_p$ independent of the dimension $d$ such that
\begin{align*}
\|M^\ast f\|_{L^p(\mathbb R^d)}\le C_p\|f\|_{L^p(\mathbb R^d)}.
\end{align*}
\end{proposition}
\begin{proof}
One has
\begin{equation*}
   h(x) \le \left\{
    \begin{aligned}
    & \frac{\Gamma\left(\frac{1+d}2\right)}{\pi^{\frac {d+3}2}}\bigg[\frac{\Gamma\left(\frac{1}2\right)\Gamma\left(\frac{1+d}2\right)}{\Gamma\left(1+\frac d2\right)}+\sum_{n=1}^\infty\frac{|x|^{2n}}{n}\bigg],\ \ |x|<1,\\
    & \frac1{|x|^{1+d}} \frac{\Gamma\left(\frac{1+d}2\right)}{\pi^{\frac {d+3}2}}\bigg[\frac{\Gamma\left(\frac{1}2\right)\Gamma\left(\frac{1+d}2\right)}{\Gamma\left(1+\frac d2\right)}+\sum_{n=1}^\infty\frac1{|x|^{2n}\sqrt{n(n+\frac d2)}}\bigg],\ \ |x|>1,
    \end{aligned}
    \right.
\end{equation*}
where we have used Gautchi's inequality \cite[5.6.4]{OLBC}
\begin{align}\label{G1}
x^{1-s}<\frac{\Gamma(x+1)}{\Gamma(x+s)}<(x+1)^{1-s},\ \ \ x>0,\ \ s\in(0,1),
\end{align} 
and
\begin{align*}
\frac1{n\sqrt{4+d}}<\frac{\Gamma\left(n+\frac{1}2\right)\Gamma\left(n+\frac{1+d}2\right)}{n!\Gamma\left(n+1+\frac d2\right)}\le \frac1{\sqrt{n(n+\frac d2)}},\ n\in\mathbb N.
\end{align*}
Define  functions
\begin{align*}
   h^0(x)&=\frac{\Gamma^2\left(\frac{1+d}2\right)}{\pi^{1+\frac {d}2}\Gamma(1+\frac d2)}\bigg(\chi_{B(0,1)}(x)+ \frac1{|x|^{1+d}}\chi_{B^c(0,1)}(x)\bigg),\\
h^1(x)&=-\frac{\Gamma(\frac{1+d}2)}{\pi^{\frac {d+3}2}}\ln(1-|x|^2)\chi_{B(0,1)}(x),\\
h^2(x)&=\frac{\Gamma(\frac{1+d}2)}{\pi^{\frac {d+2}2}|x|^{1+d}}\sum_{n=1}^\infty\frac1{|x|^{2n}\sqrt{n(n+d)}}\chi_{B^c(0,1)}(x).
\end{align*}
It follows that 
\begin{align*}
h(x)\le h^0(x)+h^1(x)+h^2(x)
\end{align*}
and
 \begin{align*}
M^\ast f(x)\le \sup_{t>0}\left|h^0_t\ast f(x)\right|+\sup_{t>0}\left|h^1_t\ast f(x)\right|+\sup_{t>0}\left|h^2_t\ast f(x)\right|.
\end{align*}
\par
Note that $h^0$ is integrable, non-increasging and its least radial decreasing majorant is itself. Moreover, 
\begin{align*}
  \|h^0\|_{L^1(\mathbb R^d)}=\frac{\Gamma^2\left(\frac{1+d}2\right)\omega_{d-1}}{\pi^{1+\frac {d}2}\Gamma(1+\frac d2)}\left(\int_0^1r^{d-1}dr+\int_1^\infty\frac1{r^2}dr\right)\le\frac{4\Gamma^2\left(\frac{1+d}2\right)\pi^{\frac d2}}{\pi^{1+\frac {d}2}\Gamma(\frac d2)\Gamma(1+\frac d2)}\le\frac6{\pi}.
\end{align*}
Therefore, we obtain
\begin{equation*}
   \sup_{t>0}\left|h^0_t\ast f(x)\right|\le\frac6{\pi}Mf(x)
\end{equation*}
and 
\begin{align*}
  \left\| \sup_{t>0}\left|h^0_t\ast f\right|\right\|_{L^p(\mathbb R^d)}\le C_p\|f\|_{L^p(\mathbb R^d)}.
\end{align*}
\par
For the second maximal operator, we note that
\begin{equation*}
   \sup_{t>0}\left|h^1_t\ast f(x)\right|\le\frac{\Gamma(\frac{1+d}2)}{\pi^{\frac {d+3}2}}M_\psi f(x).
\end{equation*}
Theorem \ref{dfppe}, \eqref{psil1} and \eqref{G1} imply that
\begin{align*}
  \left\| \sup_{t>0}|h^1_t\ast f|\right\|_{L^p(\mathbb R^d)}&\le\frac{\Gamma(\frac{1+d}2)}{\pi^{\frac {d+3}2}}\left\|M_\psi f\right\|_{L^p(\mathbb R^d)}\le C_p\frac{\Gamma(\frac{1+d}2)}{\pi^{\frac {d+3}2}}\|\psi\|_{L^1(\mathbb R^d)}\|f\|_{L^p(\mathbb R^d)}\\
&\le C_p\frac{\Gamma(\frac{1+d}2)}{\pi^{\frac {d+3}2}}\frac{6\pi^{\frac d2}\ln(d+1)}{d\Gamma(\frac d2)}\|f\|_{L^p(\mathbb R^d)}\\
&\le C_p\frac{\ln(d+1)}{\sqrt d}\|f\|_{L^p(\mathbb R^d)}.
\end{align*}
\par
To treat the last maximal operator, we divide the kernel and dominate it by the sum of four parts
\begin{align*}
h^2(x)&\le \frac{\Gamma(\frac{1+d}2)}{\sqrt d\pi^{\frac {d+2}2}|x|^{1+d}}\sum_{n=1}^d\frac1{|x|^{2n}\sqrt{n}}\chi_{B^c(0,1)}(x)+\frac{\Gamma(\frac{1+d}2)\ln(|x|)}{\pi^{\frac {d+2}2}|x|^{1+3d}}\chi_{B^c(0,1)}(x)\\
&-\frac{\Gamma(\frac{1+d}2)\ln(|x|-1)}{\pi^{\frac {d+2}2}|x|^{1+3d}}\chi_{B^c(0,2)}(x)-\frac{\Gamma(\frac{1+d}2)\ln(|x|-1)}{\pi^{\frac {d+2}2}|x|^{1+3d}}\chi_{B(0,2)\setminus B(0,1)}(x)\\
&:=h^{2,1}(x)+h^{2,2}(x)+h^{2,3}(x)+h^{2,4}(x).
\end{align*}
For $n=1,2,\cdots,d$, the least radial decreasing majorant of $|x|^{-(1+d+2n)}\chi_{B^c(0,1)}(x)$ is
\begin{equation*}
   H_n(x)=\left\{
    \begin{aligned}
    &1,\ \ |x|<1,\\
    &\frac1{|x|^{1+d+2n}},\ \ |x|>1,
    \end{aligned}
    \right.
\end{equation*}
with $\|H_n\|_{L^1(\mathbb R^d)}\le2\omega_{d-1}/n$. \eqref{G1} yields that
\begin{align*}
   \sup_{t>0}\left|h^{2,1}_t\ast f(x)\right|&\le\frac{2\omega_{d-1}\Gamma(\frac{1+d}2)}{\sqrt d\pi^{\frac {d+2}2}}\sum_{n=1}^d\frac1{n\sqrt n}Mf(x)\\
    &\le \frac{4\Gamma(\frac{1+d}2)}{\sqrt d\pi^{\frac {d+2}2}}\cdot \frac{2\pi^{\frac d2}}{\Gamma(\frac{d}2)}Mf(x)\le 2Mf(x)
\end{align*}
and
\begin{align*}
  \left\| \sup_{t>0}\left|h^{2,1}_t\ast f\right|\right\|_{L^p(\mathbb R^d)}\le C_p\|f\|_{L^p(\mathbb R^d)}.
\end{align*}
\par
Note that $h^{2,2},h^{2,3}$ have the following radial decreasing majorant
\begin{equation*}
  \frac{\Gamma(\frac{1+d}2)}{\pi^{\frac {d+2}2}}\left\{
    \begin{aligned}
    &1,\ \ |x|<1,\\
    &\frac1{|x|^{3d}},\ \ |x|>1,
    \end{aligned}
    \right.
\end{equation*}
with $L^1$ norm less than $1/\sqrt d$. Therefore,
\begin{align*}
   \sup_{t>0}\left|h^{2,2}_t\ast f(x)\right|+\sup_{t>0}\left|h^{2,3}_t\ast f(x)\right|&\le \frac2{\sqrt d}Mf(x)
\end{align*}
and
\begin{align*}
 \left\|\sup_{t>0}\left|h^{2,2}_t\ast f\right|\right\|_{L^p(\mathbb R^d)}+\left\|\sup_{t>0}\left|h^{2,3}_t\ast f\right|\right\|_{L^p(\mathbb R^d)}\le \frac{C_p}{\sqrt d}\|f\|_{L^p(\mathbb R^d)}.
\end{align*}
\par
To deal with the last maximal operator, we set $\phi(x)=-\ln\left(|x|-1\right)\chi_{B(0,2)\setminus B(0,1)}(x)$ and $\tilde{\phi}(x)=-\frac{\ln(|x|-1)}{|x|^{1+3d}}\chi_{B(0,2)\setminus B(0,1)}(x)$, then
\begin{align*}
 M_{\tilde{\phi}} f(x)\le M_{\phi} f(x).
\end{align*}
Using \eqref{dfmphi} and arguments in the proof of Theorem \ref{dfppe}, we get that
there is a positive constant $C_p$ independent of the dimension $d$ such that
\begin{align*}
\|M_{\tilde{\phi}} f\|_{L^p(\mathbb R^d)}\le C_p\|\tilde{\phi}\|_{L^1(\mathbb R^d)}\|f\|_{L^p(\mathbb R^d)}, \ 1<p\le\infty.
\end{align*}
H\"{o}lder's inequality implies 
\begin{align*}
 \|\tilde{\phi}\|_{L^1(\mathbb R^d)}&=-\omega_{d-1}\int_1^2\frac{\ln(r-1)}{r^{2d+2}}dr\le \omega_{d-1}\left(\int_1^2\ln^2(r-1)dr\right)^{\frac12}\left(\int_1^2\frac{dr}{r^{4d+4}}\right)^{\frac12}\le\frac{\omega_{d-1}}{\sqrt{d}}.
\end{align*}
Applying \eqref{G1}, we observe that 
\begin{align*}
  \left\|\sup_{t>0}\left|h^{2,4}_t\ast f\right|\right\|_{L^p(\mathbb R^d)}&\le\frac{\Gamma(\frac{1+d}2)}{\pi^{\frac {d+2}2}}\left\|M_{\tilde{\phi}} f\right\|_{L^p(\mathbb R^d)}\le C_p\frac{\Gamma(\frac{1+d}2)}{\pi^{\frac {d+2}2}}\|\tilde{\phi}\|_{L^1(\mathbb R^d)}\|f\|_{L^p(\mathbb R^d)}\\
  &\le C_p\frac{\Gamma(\frac{1+d}2)}{\pi^{\frac {d+2}2}}\frac{2\pi^{\frac {d}2}}{\Gamma(\frac d2)\sqrt{d}}\|f\|_{L^p(\mathbb R^d)}\le C_p\|f\|_{L^p(\mathbb R^d)}.
\end{align*}
This completes the proof of Proposition \ref{dfMastpp}.
\end{proof}
\par
Using \eqref{mRjM} and Proposition \ref{dfMastpp}, we obtain Theorem \ref{Rdfppe}. Arguments in Remark \ref{another proof} imply $M^\ast$ is not of weak type $(1,1)$, it follows that
 for every $1\le j\le d$ and constant $C>0$, there exists a function $f$ in $H^1(\mathbb R^d)$ such that
\begin{align*}
\|R^\ast_jf\|_{L^{1,\infty}(\mathbb R^d)}\ge C\|R_jf\|_{L^1(\mathbb R^d)}.
\end{align*}
However, Theorem \ref{dyad w11}, Theorem \ref{dyad w11phi} and the weak type $(1,1)$ estimate for Hardy-Littlewood maximal operator yield
\begin{align*}
\lambda\left|\left\{x\in\mathbb R^d:M^{D,\ast} f(x)>\lambda\right\}\right|\le C_d\|f\|_{L^1(\mathbb R^d)},\ \ \lambda>0,
\end{align*}
and
\begin{align*}
\lambda\left|\left\{x\in\mathbb R^d:R^{D,\ast}_jf(x)>\lambda\right\}\right|\le C_d\|R_jf\|_{L^1(\mathbb R^d)},\ \ \lambda>0.
\end{align*}
\par
Obviously, Proposition \ref{M2}, Proposition \ref{M3} and pointwise estimates for $M^\ast$ in the proof of Proposition \ref{dfMastpp} yield that there is a constant $C_d>0$ such that
 \begin{align*}
 M^\ast f(x)\le C_dM^2f(x)\ \ \text{and}\ \ R^\ast_jf(x)\le C_dM^2(R_jf)(x).
 \end{align*}
 \eqref{M2llogl} implies there is a constant $C_d>0$ such that
 \begin{align*}
\left|\left\{x\in\mathbb R^d:M^\ast f(x)>\lambda\right\}\right|\le C_d\int_{\mathbb R^d}\left(\frac{|f(x)|}\lambda\right)\log\left(1+\frac{|f(x)|}\lambda\right)dx,\ \ \lambda>0,
\end{align*}
and
 \begin{align*}
\left|\left\{x\in\mathbb R^d:R^\ast_jf(x)>\lambda\right\}\right|\le C_d\int_{\mathbb R^d}\left(\frac{|R_jf(x)|}\lambda\right)\log\left(1+\frac{|R_jf(x)|}\lambda\right)dx,\ \ \lambda>0.
\end{align*}
 This proves Theorem \ref{endpoint estimates Rj}.

\section{Appendix}
\begin{proposition}\label{M3}
Let $\phi(x)=-\ln\left(|x|-1\right)\chi_{B(0,2)\setminus B(0,1)}(x)$. Then there is a constant $C$ independent of the dimension $d$ such that
\begin{align*}
M_\phi f(x)\le C8^d\ln d\nu_d\ M^2f(x),\ \ x\in\mathbb R^d.
\end{align*}
\end{proposition}
\begin{proof}
By translating and dilating one reduces the proof of Proposition \ref {M3} to
\begin{align*}
\phi\ast f(0)\le C8^d\ln d\nu_d\ M^2f(0).
\end{align*}
For $1\le p<\infty$, $f\in L^p(\mathbb R^d)$,  \eqref{gHi} and \eqref{logm2} give
\begin{align*}
|\phi\ast f(0)|&=\left|\int_{B(0,2)\setminus B(0,1)}\left(\phi(y)-1\right)f(y)dy+\int_{B(0,2)\setminus B(0,1)}f(y)dy\right|\\
&\le\int_{B(0,2)\setminus B(0,1)}|\phi(y)-1|f(y)dy+|B(0,2)|Mf(0)\\
&\le |B(0,2)|\|\phi-1\|_{\exp L,B(0,2)}\|f\|_{L\log L,B(0,2)}+|B(0,2)|Mf(0)\\
&\le C2^d\ln d\nu_d\ M_{L\log L}f(0)+2^d\nu_d\ Mf(0)\\
&\le C8^d\ln d\nu_d\ M^2f(0),
\end{align*}
where we have used the fact that
\begin{align}\label{phi-1}
\|\phi-1\|_{\exp L,B(0,2)}<100\ln d.
\end{align}
Indeed, we write
\begin{align*}
\|\phi-1\|_{\exp L,B(0,2)}&=\inf\bigg\{\lambda>0:\frac1{|B(0,2)|}\int_{B(0,2)}\big(e^{\frac{|\phi(y)-1|}{\lambda}}-1\big)dy\le 1\bigg\}\\
&=\inf\bigg\{\lambda>0:\frac1{|B(0,2)|}\int_{B(0,2)}e^{\frac{|\phi(y)-1|}{\lambda}}dy\le 2\bigg\}.
\end{align*}
\par
Observe that $-\ln(s-1)$ is strictly decreasing on $(1,2]$, $\phi(y)=0$ when $|y|=2$, and $\phi(y)$ approches $+\infty$ as $|y|$ tends to $1$ from right. These facts imply that $\phi(y)-1<0$ when $y\in B(0,2)\setminus B(0,1+e^{-1})$ and $\phi(y)-1\ge0$ when $y\in B(0,1+e^{-1})\setminus B(0,1)$. It follows that
\begin{align*}
e^{\frac{|\phi(y)-1|}{\lambda}}=e^{-\frac1\lambda}\chi_{B(0,1+e^{-1})\setminus B(0,1)}(y)(|y|-1)^{-\frac1{\lambda}}+e^{\frac1\lambda}\chi_{B(0,2)\setminus B(0,1+e^{-1})}(y)(|y|-1)^{\frac1{\lambda}}.
\end{align*}
When $\lambda>2$, it is dominated by the following integrable function
\begin{align*}
\chi_{B(0,1+e^{-1})\setminus B(0,1)}(y)(|y|-1)^{-\frac12}+e^{\frac12}\chi_{B(0,2)\setminus B(0,1+e^{-1})}(y).
\end{align*}
Lebesgue's dominated convergence theorem yields
\begin{align*}
\lim_{\lambda\rightarrow+\infty}\frac1{|B(0,2)|}\int_{B(0,2)}e^{\frac{|\phi(y)-1|}{\lambda}}dy=1.
\end{align*}
Thus, there is a $\lambda_0>0$ such that 
\begin{align*}
\frac1{|B(0,2)|}\int_{B(0,2)}e^{\frac{|\phi(y)-1|}{\lambda_0}}dy\le 2.
\end{align*}
Therefore, we have
\begin{align*}
\|\phi-1\|_{\exp L,B(0,2)}\le \lambda_0<\infty.
\end{align*}
Moreover, when $d\ge3$ and $\lambda=100\ln d$, by switching to polar coordinates, we observe that
\begin{align*}
\frac1{|B(0,2)|}\int_{B(0,2)}e^{\frac{|\phi(y)-1|}{\lambda}}dy&=\frac {de^{\frac1\lambda}}{2^d}\int^2_{1+e^{-1}}(r-1)^{\frac1{\lambda}}r^{d-1}dr+\frac {de^{-\frac1\lambda}}{2^d}\int_1^{1+e^{-1}}(r-1)^{-\frac1{\lambda}}r^{d-1}dr+\frac {e^{\frac1\lambda}}{2^d}\\
&\le \frac {de^{\frac1\lambda}}{2^d}\int^2_{1+e^{-1}}r^{d-1}dr+\frac {dd^{\frac1\lambda}e^{-\frac1\lambda}}{2^d}\int^{1+e^{-1}}_{1+d^{-1}}r^{d-1}dr\\
&+\frac {de^{-\frac1\lambda}}{2^d}\left(1+\frac1d\right)^{d-1}\int_1^{1+d^{-1}}(r-1)^{-\frac1{\lambda}}dr+\frac {e^{\frac1\lambda}}{2^d} \\
&\le e^{\frac1\lambda}\left(1-\left(\frac {1+e^{-1}}2\right)^d\right)+\frac {d^{\frac1\lambda}e^{-\frac1\lambda}}{2^d}\left[\left(1+\frac1e\right)^{d}-\left(1+\frac1d\right)^{d}\right]\\
&+\frac{d^{1+\frac1\lambda}e^{1-\frac1\lambda}\lambda}{2^d(\lambda-1)(1+d)}+\frac {e^{\frac1\lambda}}{2^d}\\
&\le2,
\end{align*}
which means 
\begin{align*}
\|\phi-1\|_{\exp L,B(0,2)}\le 100\ln d.
\end{align*}
This completes the proof. 
\end{proof}
Using arguments similar to that in Theorem \ref{ppe}, Theorem \ref{dfppe}, Theorem \ref{notw11} and Remark \ref{another proof}, we obtain the following two results.
\begin{theorem}\label{ppephi}
For $1<p\le\infty$, there is a constant $C_p$ independent of the dimension $d$ such that
\begin{align}\label{dfmphi}
\|M_\phi f\|_{L^p(\mathbb R^d)}\le C_p8^d\ln d\nu_d\|f\|_{L^p(\mathbb R^d)},
\end{align}
and
\begin{align*}
\|M_\phi f\|_{L^p(\mathbb R^d)}\le C_p\|\phi\|_{L^1(\mathbb R^d)}\|f\|_{L^p(\mathbb R^d)}.
\end{align*}
\end{theorem}
\begin{theorem}\label{notw11phi}
$M_\phi$ is not of weak type $(1,1)$, i.e. for any positive constant $C$, there exists a function $f\in L^1(\mathbb R^d)$ such that
\begin{align*}
\|M_\phi f\|_{L^{1,\infty}(\mathbb R^d)}\ge C\|f\|_{L^1(\mathbb R^d)}.
\end{align*}
However, $M_\phi$ is of weak type $L\log L$, i.e.
there is a positive constant $C$ independent of the dimension $d$ such that
\begin{align*}\label{epephi}
\left|\left\{x\in\mathbb R^d:M_\phi f(x)>\lambda\right\}\right|\le C\int_{\mathbb R^d}\Phi\left(\frac{|f(x)|}\lambda\right)dx,\ \ \lambda>0,
\end{align*}
where $\Phi(s)=s\left(1+\log^+s\right)$.
\end{theorem}
We also proved the following result.
\begin{theorem}\label{dyad w11phi}
$M^D_\phi$ is  of weak type $(1,1)$, i.e. there is a positive constant $C_d$ dependent on the dimension $d$ such that
\begin{align*}
\lambda\left|\left\{x\in\mathbb R^d:M^D_\phi f(x)>\lambda\right\}\right|\le C_d\|f\|_{L^1(\mathbb R^d)},\ \ \lambda>0.
\end{align*}
\end{theorem}
\begin{proof}
Theorem \ref{ppephi} implies that $M_\phi^D$ is bounded on $L^p(\mathbb R^d)$ for $1<p\le\infty$. We apply vector-valued estimates for the sublinear operator (see \cite[Corollary 5.6.5]{G14}), it suffices to verify the following regularity condition
\begin{align*}
\sup_{y\neq0}\int_{|x|\ge2|y|}\sup_{j\in\mathbb Z}\left|\phi_{2^j}(x-y)-\phi_{2^j}(x)\right|dx<\infty.
\end{align*}
\par
Without loss of generality, we assume that $y=(|y|,0,\cdots,0)$ and $|y|\neq0$. Then, there is a $k_0\in\mathbb Z$ such that  $|y|\in[2^{k_0-1},2^{k_0})$. For $|x|\ge2|y|$, we observe that $x_1>\frac{|y|}2$ if and only if $|x|>|x-y|$.
\par
Set $R_{k_0}=\left\{x:2|y|\le |x|<2^{k_0+1}\right\}$, $R_k=\left\{x:2^k\le |x|<2^{k+1}\right\}$ for $k\ge k_0+1$. We write
\begin{align*}
\int_{|x|\ge2|y|}\sup_{j\in\mathbb Z}\left|\phi_{2^j}(x-y)-\phi_{2^j}(x)\right|dx=\sum_{k\ge k_0}\int_{R_k}\sup_{j\in\mathbb Z}\left|\phi_{2^j}(x-y)-\phi_{2^j}(x)\right|dx.
\end{align*}
\par
We work with summands for $k\ge k_0+1$, the summand with $k=k_0$ can be treated similarly. Clearly, for $x\in R_k$, $\phi_{2^j}(x)=0$ when $j\neq k$, and $\phi_{2^j}(x-y)=0$ when $|j-k|\ge2$.
\par
To estimate each summand, we divide each $R_k$ into four parts: 
$$R_{k,1}:=\left\{x\in R_k:x_1>\frac{|y|}2,x\notin B(y,2^{k})\right\},$$
$$R_{k,2}:=\left\{x\in R_k:x_1>\frac{|y|}2,x\in B(y,2^{k})\right\},$$
$$R_{k,3}:=\left\{x\in R_k:x_1\le\frac{|y|}2,x\in B(y,2^{k+1})\right\},$$
$$R_{k,4}:=\left\{x\in R_k:x_1\le\frac{|y|}2,x\notin B(y,2^{k+1})\right\}.$$
\par
For each summand, we write it as the sum of four terms
\begin{align*}
\int_{R_k}\sup_{j\in\mathbb Z}\left|\phi_{2^j}(x-y)-\phi_{2^j}(x)\right|dx&=\sum_{i=1}^4\int_{R_{k,i}}\sup_{j\in\mathbb Z}\left|\phi_{2^j}(x-y)-\phi_{2^j}(x)\right|dx.
\end{align*}
\par
If $x\in R_{k,1}$, then $2^{k}\le|x-y|\le|x|<2^{k+1}$.  For $j=k-1,k+1$, $\phi_{2^j}(x-y)=0$. Thus,
\begin{align*}
\sup_{j\in\mathbb Z}\left|\phi_{2^j}(x-y)-\phi_{2^j}(x)\right|=\phi_{2^k}(x-y)-\phi_{2^k}(x)=\frac1{2^{kd}}\ln\left(1+\frac{|x|-|x-y|}{|x-y|-2^k}\right).
\end{align*}
Note that $|x|-|x-y|\le |y|$. Therefore, using the change of variables we deduce that
\begin{align*}
\int_{R_{k,1}}\sup_{j\in\mathbb Z}\left|\phi_{2^j}(x-y)-\phi_{2^j}(x)\right|dx&\le\frac1{2^{kd}}\int_{2^k\le|x-y|<2^{k+1}}\ln\left(1+\frac{|y|}{|x-y|-2^k}\right)dx\\
&\le\frac{2^{d-1}\omega_{d-1}|y|}{2^{k}}\int_{\frac {|y|}{2^k}}^\infty\frac{\ln(1+s)}{s^2}ds\\
&\le 3\cdot2^{d}\omega_{d-1}\frac{k-k_0}{2^{k-k_0}}.
\end{align*}
\par
If $x\in R_{k,2}$, then $2^k-|y|\le|x-y|<2^k$. For $j=k,k+1$, $\phi_{2^j}(x-y)=0$. 
Thus,
\begin{align*}
\sup_{j\in\mathbb Z}\left|\phi_{2^j}(x-y)-\phi_{2^j}(x)\right|=\max\{\phi_{2^{k-1}}(x-y),\phi_{2^k}(x)\}\le\phi_{2^{k-1}}(x-y)+\phi_{2^k}(x).
\end{align*}
Therefore, using the change of variables we have
\begin{align*}
\int_{R_{k,2}}\sup_{j\in\mathbb Z}\left|\phi_{2^j}(x-y)-\phi_{2^j}(x)\right|dx&\le-\frac1{2^{(k-1)d}}\int_{2^k-|y|\le|x-y|<2^k}\ln\left(\frac{|x-y|}{2^{k-1}}-1\right)dx\\
&-\frac1{2^{kd}}\int_{2^k\le|x|<2^k+|y|}\ln\left(\frac{|x|}{2^{k}}-1\right)dx\\
&\le-2^{d-1}\omega_{d-1}\left(\int_{2-\frac{|y|}{2^{k-1}}}^2\ln(s-1)ds+\int^{1+\frac{|y|}{2^k}}_1\ln(s-1)ds\right)\\
&\le 3\cdot2^d\omega_{d-1}\frac{k-k_0}{2^{k-k_0}}.
\end{align*}
\par
 If $x\in R_{k,3}$, then $2^k\le|x|\le|x-y|<2^{k+1}$. For $j=k-1,k+1$, $\phi_{2^j}(x-y)=0$. As in the case $x\in R_{k,1}$ we have
\begin{align*}
\sup_{j\in\mathbb Z}\left|\phi_{2^j}(x)-\phi_{2^j}(x-y)\right|&=\phi_{2^k}(x)-\phi_{2^k}(x-y)=\frac1{2^{kd}}\ln\left(1+\frac{|x-y|-|x|}{|x|-2^k}\right).
\end{align*}
Note that $|x-y|-|x|\le |y|$. Therefore, using the change of variables we deduce that
\begin{align*}
\int_{R_{k,3}}\sup_{j\in\mathbb Z}\left|\phi_{2^j}(x-y)-\phi_{2^j}(x)\right|dx&\le\frac1{2^{kd}}\int_{2^k\le|x|<2^{k+1}}\ln\left(1+\frac{|y|}{|x|-2^k}\right)dx\\
&\le\frac{2^d\omega_{d-1}|y|}{2^{k+1}}\int_{\frac {|y|}{2^k}}^\infty\frac{\ln(1+s)}{s^2}ds\\
&\le 3\cdot 2^d\omega_{d-1}\frac{k-k_0}{2^{k-k_0}}.
\end{align*}
\par
 If $x\in R_{k,4}$, then $2^{k+1}-|y|\le|x|<2^{k+1}$. For $j=k-1,k$, $\phi_{2^j}(x-y)=0$. Thus,
\begin{align*}
\sup_{j\in\mathbb Z}\left|\phi_{2^j}(x-y)-\phi_{2^j}(x)\right|=\max\left\{\phi_{2^{k+1}}(x-y),\phi_{2^k}(x)\right\}\le\phi_{2^{k+1}}(x-y)+\phi_{2^k}(x).
\end{align*}
Therefore, using the change of variables we deduce that
\begin{align*}
\int_{R_{k,4}}\sup_{j\in\mathbb Z}\left|\phi_{2^j}(x-y)-\phi_{2^j}(x)\right|dx&\le-\frac1{2^{(k+1)d}}\int_{2^{k+1}\le|x-y|<2^{k+1}+|y|}\ln\left(\frac{|x-y|}{2^{k+1}}-1\right)dx\\
&-\frac1{2^{kd}}\int_{2^{k+1}-|y|\le|x|<2^{k+1}}\ln\left(\frac{|x|}{2^k}-1\right)dx\\
&\le-2^{d-1}\omega_{d-1}\left(\int^{1+\frac{|y|}{2^{k+1}}}_1\ln(s-1)ds+\int_{2-\frac{|y|}{2^k}}^2\ln(s-1)ds\right)\\
&\le 3\cdot2^d\omega_{d-1}\frac{k-k_0}{2^{k-k_0}}.
\end{align*}
\par
Finally, combining above four estimates, we obtain
\begin{align*}
\int_{|x|\ge2|y|}\sup_{j\in\mathbb Z}\left|\phi_{2^j}(x-y)-\phi_{2^j}(x)\right|dx&=\sum_{k\ge k_0}\int_{R_k}\sup_{j\in\mathbb Z}\left|\phi_{2^j}(x-y)-\phi_{2^j}(x)\right|dx\\
&\le 12\cdot2^d\omega_{d-1}\sum_{k\ge k_0+1}\frac{k-k_0}{2^{k-k_0}}\le 24\cdot2^d\omega_{d-1}.
\end{align*}
\end{proof}

\newpage
\begin{center}
Declarations
\end{center}
\section*{Funding}
The first author was partially supported by grants RYC2021-032600-I and PID2022-141354NB-I00. All authors were supported  by NSF of China grant 12571101.
\section*{Conflicts of interest/Competing interests}
On behalf of all authors, the corresponding author states that there is no conflict of interest.
\section*{Availability of data and material}
Not applicable.
\section*{Code availability}
Not applicable.
\end{document}